\documentclass[10 pt]{amsart}
\usepackage{fullpage, amsthm, amsmath, amsfonts, amssymb, url, mathrsfs, stmaryrd}
\usepackage{pxfonts}

\author{Casey Rodriguez}
\address{Department of Mathematics\\
  University of Chicago\\
  5734 S. University Avenue \\
  Chicago, IL 60637}
\email{c-rod216@math.uchicago.edu}

\numberwithin{equation}{section}

\newcommand{\R}{\mathbb R}

\newcommand{\Z}{\mathbb Z}

\newcommand{\ra}{\rangle}
\newcommand{\la}{\langle}

\newcommand{\rar}{\rightarrow}

\newcommand{\energysp}{\dot{H}^{s_p} \times \dot{H}^{s_p-1}}

\newcommand{\supp}{\mbox{supp }}

\newcommand{\laplace}{\Delta}

\newtheorem{lem}{Lemma}[section]
\newtheorem{thm}[lem]{Theorem}
\newtheorem{ppn}[lem]{Proposition}
\newtheorem{defn}[lem]{Definition}
\newtheorem{clm}[lem]{Claim}

\title{Scattering for radial energy--subcritical \\ wave equations in dimensions 4 and 5}
\date{\today}

\begin{document}

\begin{abstract}
In this paper, we consider the focusing and defocusing energy--subcritical, nonlinear wave equation in $\R^{1+d}$ with radial initial data
for $d = 4, 5$. We prove
that if a solution remains bounded in the critical space on its interval of existence, then the solution exists globally
and scatters at $\pm \infty$. The proof follows the concentration compactness/rigidity method initiated by 
Kenig and Merle, and the main obstacle is to show the nonexistence of 
nonzero solutions with a certain compactness property.  A main novelty of this work is the use of a simple virial argument to rule out the existence
of nonzero solutions with this compactness property rather than channels of energy arguments that have been proven to be most  
useful in odd dimensions. 
\end{abstract}

\maketitle

\section{Introduction}
In this paper we consider the following nonlinear wave equation on $\R^{1+d}$,
\begin{align}\label{nlw}
\left\{
     \begin{array}{lr}
       \partial_t^2 u - \Delta u + \mu |u|^{p-1} u = 0, \quad (t,x) \in \R \times \R^d \\
       \vec u(0) = (u_0,u_1) \in \energysp
     \end{array}
   \right..
\end{align}
Here $\mu \in \{ \pm 1\}$, $s_p = \frac{d}{2} - \frac{2}{p-1}$, and $\vec u(t)$ is the vector
$$
\vec u(t) := (u(t), \partial_t u(t)) \in \energysp.
$$
Throughout this paper, we will often write $u(t)$ instead of the function $u(t,x)$. 

The case $\mu = 1$ is referred to as the \emph{defocusing} case, and the case $\mu = -1$ is referred to as the \emph{focusing} case.  The following 
energy is conserved along the flow for regular enough initial data
\begin{align*}
E(\vec u(t)) := \frac{1}{2} \int |\nabla u(t)|^2 + |\partial_t u(t)|^2 dx + \mu \frac{1}{p+1} \int |u(t)|^{p+1} dx
 = E(\vec u(0)). 
\end{align*}
Note that for regular enough initial data, the energy is nonnegative in the defocusing case $\mu = 1$ but 
can be negative in the focusing case $\mu = -1$. 

The Cauchy problem \eqref{nlw} is \emph{critical} in the space $\energysp$ since the scaling
\begin{align*}
u(t,x) \mapsto \frac{1}{\lambda^{\frac{2}{p-1}}} u \left ( \frac{t}{\lambda^{\frac{2}{p-1}}} 
, \frac{x}{\lambda^{\frac{2}{p-1}}}\right )
\end{align*}
preserves both the equation \eqref{nlw} and the size of the initial data $\| (u_0,u_1) \|_{\energysp}$.  If the initial 
regularity $s_p$ 
lies below the regularity of the conserved energy $E(\vec u)$ (i.e. $s_p < 1$), we say
that \eqref{nlw} is \emph{energy--subcritical}.  If the regularity satisfies $s_p > 1$, then we say that \eqref{nlw}
is \emph{energy--supercritical}.  In the special case $s_p = 1$, we say that \eqref{nlw} is \emph{energy--critical}.

The local well--posedness of \eqref{nlw} is well understood for $3 \leq d \leq 5$ and $s_p \geq 1/2$ (see for example \cite{lind}).  Indeed, using Strichartz estimates and the chain rule for fractional 
derivatives, given any initial data $(u_0, u_1) \in \energysp$, there exists a unique solution $u$ to \eqref{nlw}
defined on a maximal time interval $I_{\max}(u) = (T_-(u), T_+(u))$.  Moreover, if the size of the initial data $\| (u_0, u_1) \|_{\energysp}$
is sufficiently small, then $I_{\max}(u) = \R$ and $u$ scatters, i.e. there exist solutions $v_L^{\pm}$ to the 
free wave equation $\partial_t^2 v_L^{\pm} -
\laplace v_L^{\pm} = 0$ so that
\begin{align}\label{scatter}
\lim_{t \rightarrow \pm \infty} \| \vec u(t) - \vec{v}_L^{\pm}(t) \|_{\energysp} = 0.
\end{align}

Much less is known about the long time asymptotics for \eqref{nlw} for large data.  In the focusing case $\mu = -1$, one can have finite time blow--up
even for smooth compactly supported initial data.  To see this, we note that 
\begin{align*}
v(t) = \left (\frac{2(p+1)}{(p-1)^2} \right )^{\frac{1}{p-1}}  (T-t)^{-\frac{2}{p-1}}
\end{align*}
satisfies $\partial_t^2 v(t) - v(t)^p = 0$ and blows up at $t = T$. Let $\varphi$ be a smooth compactly supported cutoff with 
$\varphi(x) = 1$ for $|x| \leq 2T$, and let $u$ be the solution to \eqref{nlw} with initial data $\varphi \vec v(0) 
\in C^\infty_0(\R^d) \times C^\infty_0(\R^d)$.  By 
finite speed of propagation, 
$$
|x| \leq 2T - |t| \implies u(t,x) = v(t), 
$$  
so that $u$ blows--up in finite time.  In particular, by the Sobolev embedding 
\begin{align*}
\| \vec u(t) \|_{\energysp} \geq \| u(t) \|_{L_x^{d(p-1)/2}} \geq \| v(t) \|_{L_x^{d(p-1)/2}(|x| \leq T)}, 
\end{align*}
we see that $\limsup_{t \rightarrow T} \| \vec u(t) \|_{\energysp} = +\infty$. Such blow--up is referred to as 
\emph{type I} blow--up or ODE blow--up. 

In this work, we consider \emph{type II} solutions to \eqref{nlw}, i.e. solutions $u$ that satisfy
\begin{align*}
\sup_{t \in I_{\max}(u)} \| \vec u(t) \|_{\energysp} < +\infty.
\end{align*}
We now give a very brief review of recent work dedicated to type II solutions to \eqref{nlw}. 

In the energy--critical case $s_p = 1$, there has been extensive work recently on understanding type II behavior.  For the defocusing
case, conservation of the energy automatically implies that any finite energy solution to \eqref{nlw} is a type II solution.
Moreover, it is easy to see by integration by parts that there are no nonzero stationary solutions to \eqref{nlw}, i.e. if $f \in \dot H^1$ solves 
$$
-\Delta f + |f|^{\frac{4}{d-2}} f = 0,
$$
then $f = 0$.  Thus, there is no a priori obstruction to scattering, and indeed it has been proven for $3 \leq d \leq 5$ that 
every finite energy solution to \eqref{nlw} in the defocusing case exists globally and scatters (see \cite{stru}
\cite{gril}). For the focusing case, there is an explicit stationary solution to \eqref{nlw} given by
the ground state
\begin{align*}
W(x) = \left ( 1 + \frac{|x|^2}{d(d-2)} \right )^{-\frac{d-2}{2}}.
\end{align*}
Hence there exists a globally defined type II solution for \eqref{nlw} in the focusing case that does not scatter.  Moreover, 
in the works \cite{krst} \cite{krs} \cite{hillr} \cite{jend} type II finite time blow--up solutions have been constructed for $3 \leq d \leq 5$.  The blow--up
in these works occurs via a concentration of energy at the tip of a light cone through a bubbling off of the ground state $W$.
The bubbling behavior exhibited in these constructions is believed to be characteristic of type II solutions to \eqref{nlw} in the focusing case
that either blow--up in finite time or do not scatter. This asymptotic decoupling of the solution into a sum of dynamically rescaled
ground states and a free radiation term has been proven to hold for \emph{radial} solutions
for all time approaching the final time of existence in the seminal work \cite{dkm4} in $3d$ and along a sequence of times in higher dimensions in 
\cite{ckls} \cite{me} \cite{jia}. 

Type II solutions in the energy--subcritical and supercritical regimes exhibit remarkably different asymptotics.  For the defocusing supercritical wave equation, it has been shown that type II solutions 
are global and scatter for a large variety of powers and dimensions using Morawetz type estimates 
(see for example \cite{km11} \cite{km112} \cite{kill1} \cite{kill2} \cite{bul1} \cite{bul2}).  
In the focusing subcritical/supercritical regime, very little is known.  It was shown in \cite{dkm5} (in $3d$) and in \cite{dodl2} (in $5d$) that radial type II solutions to the focusing supercritical wave equation are global and scatter.  
For the focusing subcritical wave equation, 
it was shown in \cite{shen} (for $s_p \in (1/2,1)$) and \cite{dodl1} (for $s_p = 1/2$) that radial type II solutions in $3d$ are global and scatter.  

In this work, we study radial type II solutions to the energy--subcritical wave equation in dimensions $4$ and $5$.  In particular, 
we show that every radial type II solution is global and scatters.
The precise statement of our main result is as follows.    

\begin{thm}\label{thm1}
Assume that $1/2 \leq s_p < 1$ and $d = 4,5$.  Let $u$ be a radial solution to \eqref{nlw} defined on its maximal interval of existence $I_{\max}(u) = (T_-,T_+)$.  Suppose in addition that
\begin{align}
\sup_{t \in I_{\max}(u)} \| \vec u(t) \|_{\energysp} < +\infty.
\end{align}
Then $u$ is globally defined and scatters, i.e. $I_{\max}(u) = \R$ and there exist radial solutions $v_L^{\pm}$ to the free wave equation $\partial_t^2 v_L^{\pm} -
\laplace v_L^{\pm} = 0$ so that
\begin{align}\label{scatter}
\lim_{t \rightarrow \pm \infty} \| \vec u(t) - \vec v_L^{\pm}(t) \|_{\energysp} = 0.
\end{align}
\end{thm}

We remark that the restriction on $s_p$ and $d$ is only for the purpose of having a readily available well--posedness theory.  However,
much of our work is independent of regularity and dimension.  Moreover, the methods in this work give a unified proof of 
scattering in the subcritical regime $1/2 \leq s_p < 1$ in $3d$ which was shown in the works \cite{shen} ($1/2 < s_p < 1$) and 
\cite{dodl1} (for $s_p = 1/2$).  These works used different methods for $s_p > 1/2$ and $s_p = 1/2$ that were specific to the $3d$ setting (see the discussion below for more
details).   

The proof follows the powerful concentration compactness/rigidity method initiated by Kenig and Merle in \cite{km06} and \cite{km08}, and the outline
is as follows.  The proof proceeds by contradiction.  Using the local well--posedness theory for \eqref{nlw} and the profile decompositions of Bahouri--Gerard \cite{bahger},
it is now standard to show that the failure of Theorem \ref{thm1} implies that there exists a \emph{nonzero} radial solution $u$ to \eqref{nlw}
with $I_{\max}(u) = (T_-,+\infty)$ and a continuous function $N : (T_-, +\infty) \rightarrow (0,+\infty)$ such that the trajectory
\begin{align}
K = \left \{ \left ( \frac{1}{N(t)^{\frac{2}{p-1}}} u\left (t, \frac{\cdot}{N(t)} \right ) ,
\frac{1}{N(t)^{\frac{2}{p-1}+1}} \partial_t u\left (t, \frac{\cdot}{N(t)} \right )  \right ) : t \in I \right \}
\end{align}
is pre--compact in $\energysp$.  Such solutions are said to have the compactness property. 
 Moreover, we can without loss of generality assume that either $\liminf_{t \rightarrow +\infty}
N(t) = 0$ (frequency cascade case) or $I_{\max} = \R$ and $N(t) \equiv 1$ (soliton--like case). The majority of the paper is 
dedicated to showing that in either case, we must
have $u \equiv 0$ which is a contradiction.

In order to be able to use the conservation of energy or monotonicity identities to show that $u \equiv 0$, we first show that $\vec u(t)$ has
more regularity than $\energysp$. In fact we show that $\vec u(t)$ is  
in $\dot H^1 \times L^2$ for all $t$.  This step uses the ``double Duhamel trick'' that exploits the compactness of $K$ and
is a refinement of the analogous procedure from \cite{dodl1} for the $3d$ cubic equation.  This method 
is rooted in the work by Tao \cite{taosolo} and was used in the study of nonlinear Schr\"odinger equations
\cite{killnls1} \cite{killnls2} and semilinear wave equations \cite{bul1} \cite{bul2} \cite{klls1}.  The bound in time on the $\dot H^1 \times
 L^2$ norm of $\vec u(t)$ that we obtain and 
the subcritical assumption $s_p < 1$ show that for the frequency cascade case, the energy of the compact solution 
must be nonpositive.  By a result from \cite{kill}, for energy--subcritical equations this implies that either $u \equiv 0$
or $-\infty < T_- < T_+ < \infty$.  The latter outcome contradicts the fact that $T_+ = \infty$.  Thus, $u \equiv 0$
in the frequency cascade case.  We emphasize here that our methods for ruling out the frequency cascade case rely crucially on 
the subcritical assumption $s_p < 1$ (see the comments after Proposition \ref{prop3} for the role of $s_p$ in the proof of 
the blow--up result in \cite{kill}).  For $s_p \geq 1$, different methods must be employed to rule out even the subcase of self--similar blow--up, $T_- < \infty$
and $N(t) = (t- T_-)^{-1}$ (see, for example, \cite{km08} \cite{dkm5} \cite{dodl2} \cite{dkm7}).

In the soliton--like case, we show that $\vec u(t) \in \dot H^{1 + s_p} \times \dot H^{s_p}$ with a uniform bound in time.  
This implies that the trajectory $K = \{ \vec u(t) : t \in \R \}$ is precompact in $(\energysp) \cap (\dot H^1 \times L^2)$.  In the
defocusing case, it is then simple to show using conservation of energy that $u = 0$.  In the focusing case, this is not so, and this is where the main novelty 
of our work lies.  In the works \cite{dkm5} \cite{shen} \cite{dodl2} on the focusing supercritical and subcritical wave equation in $3d$ and $5d$, the soliton--like solution
is shown to be 0 using channels of energy arguments.  These arguments require exterior energy estimates for the free wave equation that 
are strongest in odd dimensions (see \cite{cks} \cite{klls1}).  In \cite{dodl1}, the soliton--like solution is ruled out using a virial argument
that is useful only for the case $s_p = \frac{1}{2}$ along with a result from \cite{kill} for solutions 
to the subcritical wave equation with nonpositive energy.  In our work, we give a simple virial argument to rule
out soliton--like solutions that is inspired by the work on the focusing critical wave equation \cite{dkm1} and is valid for any dimension and subcritical power $p$.   
The idea is to show that the time average of $\| u(t) \|_{L_x^{p+1}}^{p+1}$ for the soliton--like solution tends to 0 which along with compactness 
implies that the soliton--like solution converges to 0 in $\energysp$ as $t \rightarrow +\infty$.  By the
local theory, we conclude that $u \equiv 0$ as desired. 

The organization of this paper is as follows.  In Section 2, we gather preliminary facts from harmonic analysis, the local well--posedness
theory for \eqref{nlw}, and concentration compactness.  In Section 3, we show that 
a solution $u$ to \eqref{nlw} with the compactness property is in the energy space $\dot H^1 \times L^2$ and satisfies
an estimate that rules out the frequency cascade case.  In Section 
4, we rule out the soliton--like case by establishing even more regularity for the soliton--like solution and by using a simple virial argument. 

\textbf{Acknowledgments:}  This work was completed during the author's doctoral studies at the University of Chicago.  The author would like
to thank his adviser, Carlos Kenig, for introducing him to nonlinear dispersive equations and for his invaluable patience and guidance.  The author would also like to thank the anonymous referees whose careful reading and suggestions greatly improved the exposition in this work.

\section{Preliminaries}

In this section, we gather some well known facts from harmonic analysis, the local well--posedness of \eqref{nlw} and the 
concentration compactness method that will be relevant to the proof of Theorem \ref{thm1}.  
\subsection{Harmonic analysis}

We first recall the following Bernstein's inequalities.  Throughout this work, the operators $P_{\leq N}$, $P_{\geq N}$, and 
$P_N$ are the standard smooth frequency cutoffs.  
For notational convenience, we will denote the Littlewood--Paley cutoffs $P_{2^k}$ simply by $P_k$.  

\begin{lem}
Let $1 \leq p \leq q \leq +\infty$, and let $s \geq 0$.  Let $f : \R^d \rightarrow \R$.  Then 
\begin{align*}
&\| P_{\geq N} f \|_{L^p} \lesssim N^{-s} \| (-\Delta)^{s/2} P_{\geq N} f \|_{L^p}, \quad
\| P_{\leq N} (-\Delta)^{s/2} f \|_{L^p} \lesssim N^{s} \| P_{\leq N} f \|_{L^p}, \\
&\| (-\Delta)^{\pm s/2} P_{N} f \|_{L^p} \simeq N^{\pm s} \| P_{N} f \|_{L^p}, \quad
\| P_{\leq N} f \|_{L^q} \lesssim N^{\frac{d}{p} - \frac{d}{q}} \| P_{\leq N} f \|_{L^p}, \\
&\| P_{N} f \|_{L^q} \lesssim N^{\frac{d}{p} - \frac{d}{q}} \| P_{N} f \|_{L^p}. 
\end{align*}
\end{lem}

In order to show a gain of regularity, we will need the notion of a frequency envelope (see \cite{taobook}). 

\begin{defn}
A sequence of positive real numbers $\alpha = \{\alpha_k \}_{k \in \Z}$ is said to be a \emph{frequency envelope}
if $\| \alpha \|_{\ell^2(\Z)} < +\infty$.  

We say that $(f,g) \in \dot H^s \times \dot H^{s-1}$ \emph{lies underneath} the frequency envelope $\alpha$ if 
\begin{align*}
\| (P_k f, P_k g) \|_{\dot H^s \times \dot H^{s-1}} \leq \alpha_k, \quad \forall k \in \Z. 
\end{align*}
\end{defn}

We note that if $(f,g)$ lies underneath $\alpha$, then 
\begin{align}
\| (f,g) \|_{\dot H^s \times \dot H^{s-1}} \lesssim \| \alpha \|_{\ell^2},
\end{align}
or more generally, for any $\sigma \in \R$
\begin{align*}
\| (f,g) \|_{\dot H^{s+\sigma} \times \dot H^{s+\sigma-1}} 
\lesssim \left \| \{ 2^{\sigma k} \alpha_k \}_{k \in \Z} \right \|_{\ell^2}. 
\end{align*}

We will also need the following refined radial Sobolev embedding which is a consequence of the Hardy--Littlewood--Sobolev
inequality (see \cite{tao} Corollary A.3). 

\begin{lem}
Let $0 < s < d$ and suppose that $f \in \dot H^s$ is radial.  Suppose that 
\begin{align}
\beta > - \frac{d}{q}, \quad \frac{1}{2} - s \leq \frac{1}{q} \leq \frac{1}{2}, \quad \frac{1}{q} = \frac{1}{2} - \frac{\beta + s}{d}, 
\end{align}
and at most one of the equalities $q = 1$, $q = \infty$, $\frac{1}{q} + s = \frac{1}{2}$, holds.  Then 
\begin{align*}
\| |x|^{\beta} f \|_{L^q} \leq C \| f \|_{\dot H^s}. 
\end{align*}
\end{lem}

A crucial tool in the local well--posedness theory and in dealing with nonlinearities that have nonintegral powers is the following chain rule and Leibniz rule for fractional derivatives
(see, for example, \cite{kpv}): 

\begin{ppn}
If $F \in C^2$, $F(0) = F'(0) = 0$, and for all $a,b$ we have $|F'(a+b)| \lesssim |F'(a)| + |F'(b)|$ and 
$|F''(a+b)| \lesssim |F''(a)| + |F''(b)|$, then we have for $0 < \alpha < 1$
\begin{align*}
\| |D_x|^\alpha F(u) \|_{L^q_x} \lesssim \| F'(u) \|_{L^{q_1}_x} \| |D_x|^{\alpha} u \|_{L^{q_2}_x}, 
\end{align*}
and
\begin{align*}
\| |D_x|^{\alpha} (F(u) - &F(v)) \|_{L^q_x} \lesssim
\left ( 
\| F'(u) \|_{L^{q_1}_x} + 
\| F'(v) \|_{L^{q_1}_x}
\right ) \| |D_x|^{\alpha} (u - v) \|_{L^{q_2}_x} \\
&+  \left ( 
\| F''(u) \|_{L^{r_1}_x} + 
\| F''(v) \|_{L^{r_1}_x}
\right ) 
\left ( 
\| |D_x|^{\alpha} u \|_{L^{r_2}_x} + 
\| |D_x|^{\alpha} v \|_{L^{r_2}_x}
\right ) \| (u - v) \|_{L^{r_3}_x}, 
\end{align*}
where $
\frac{1}{q} = \frac{1}{q_1} + \frac{1}{q_2} =
\frac{1}{r_1} + \frac{1}{r_2} + \frac{1}{r_3}
$.
\end{ppn}

Finally, we will need the following Strichartz estimates for solutions to the inhomogeneous wave equation in $\R^{1+d}$
(see \cite{gin} \cite{keel}).  We say that 
a pair $(q,r)$ of exponents is \emph{wave admissible} if $q, r \in [2,\infty]$, $(q,r,d) \neq (2,\infty,3)$, and 
\begin{align*}
\frac{1}{q} + \frac{d-1}{2r} \leq \frac{d-1}{4}.
\end{align*}

\begin{ppn}
Let $(q,r)$ and $(a,b)$ be wave admissible.  Let $s, \rho, \mu \in \R$ satisfy the scaling condition
\begin{align*}
\frac{1}{q} + \frac{d}{r} &= \frac{d}{2} - s + \rho, \\
\frac{1}{a} + \frac{d}{b} &= \frac{d-2}{2} + s + \mu.
\end{align*}
Let $I$ be an interval with $0 \in I$, and suppose $u$ is a (weak) solution to 
\begin{align*}
\left\{
     \begin{array}{lr}
       \partial_t^2 u - \Delta u  = F, \quad (t,x) \in I \times \R^d \\
       \vec u(0) = (u_0,u_1) \in \dot H^s \times \dot H^{s-1}
     \end{array}
   \right..
\end{align*}
Then 
\begin{align*}
\sup_{t \in I} \| \vec u(t) \|_{\dot H^s \times \dot H^{s-1}} + \| |D_x|^{\rho} u \|_{L^q_t L^r_x(I \times \R^d)} 
\lesssim \left ( \| \vec u(0) \|_{\dot H^s \times \dot H^{s-1}} 
+ \| |D_x|^{-\mu} F \|_{L^{a'}_t L^{b'}_x(I \times \R^d)} \right ).
\end{align*}
\end{ppn}

\subsection{Local well--posedness} 

We first define what we mean when we say that a function $u$ is a solution of \eqref{nlw}.  Let $S(t)$ be the propagator to the free wave equation, so that 
\begin{align*}
S(t)(u_0,u_1) = \cos (t \sqrt{-\Delta} ) u_0 + \frac{\sin(t \sqrt{-\Delta}) }{\sqrt{-\Delta}} u_1.
\end{align*}
We define the following function spaces 
\begin{align*}
S(J) &= L_{t,x}^{\frac{(p-1)(d+1)}{2}}(J \times \R^d), \\
W(J) &= L_{t,x}^{\frac{2(d+1)}{d-1}}(J \times \R^d), \\
N(J) &= L_{t,x}^{\frac{2(d+1)}{d+3}}(J \times \R^d).
\end{align*}
We remark that if $s_p = 1/2$, then $p = \frac{d+3}{d-1}$ and $S(J) = W(J)$.  

\begin{defn}
A solution to \eqref{nlw} on an interval $I$, where $0 \in I$, is a function $u$ such that $\vec u(t) \in C(I ; \energysp)$,
\begin{align*}
J \Subset I \implies \| u \|_{S(J)} + \| |D_x|^{s_p-1/2} u \|_{W(J)} < +\infty, 
\end{align*}
and $u(t)$ satisfies the Duhamel formulation
\begin{align*}
u(t) = S(t)(u_0,u_1) - \mu \int_0^t \frac{\sin((t-s) \sqrt{-\Delta})}{\sqrt{-\Delta}} |u(s)|^{p-1} u(s) ds.
\end{align*}
\end{defn}
 
Using the Strichartz estimates 
for the inhomogeneous equation
\begin{align*}
\sup_{t \in I} \| \vec u(t) \|_{\energysp} + \| u \|_{S(I)}& + \| |D_x|^{s_p-1/2} u \|_{W(I)} \\
&\lesssim \| (u_0,u_1) \|_{\energysp} + \| |D_x|^{s_p - 1/2} F \|_{N(I)},
\end{align*}
and the chain rule for fractional derivatives, it is now standard (via contraction and continuity arguments) to establish the following 
local well--posedness and long--time perturbation theory for \eqref{nlw} in our current setting (see for example
\cite{km06} \cite{km08} \cite{km10} \cite{km11} for details).

\begin{ppn}[Local well--posedness theory]
Let $d = 4, 5$, and let $1/2 \leq s_p < 1$. Let $I$ be an interval with $0 \in I$, and let $(u_0,u_1) \in \energysp$
with $\| (u_0,u_1) \|_{\energysp} \leq A$.  There exists 
$\delta = \delta(A) > 0$ such that if
$$
\| S(t)(u_0,u_1) \|_{S(I)} < \delta,
$$
then there exists a unique solution $u$ to \eqref{nlw} in $I \times \R^d$.  The solution $u$ satisfies the estimates
\begin{align*}
\sup_{t \in I} \| \vec u(t) \|_{\energysp} + \| |D_x|^{s_p - 1/2} u \|_{W(I)} &\leq C(A), \\
\| u \|_{S(I)} &\leq 2  \delta.
\end{align*}
Moreover, the map $(u_0, u_1) \mapsto \vec u \in C(I ; \energysp)$ is Lipschitz.  
Finally, we have persistence of regularity: if $\mu \in [s_p, s_p +1]$ and $(u_0,u_1) \in (\energysp) \cap (\dot H^{\mu} \times \dot H^{\mu-1})$, then $\vec u(t) \in C(I; (\energysp) \cap (\dot H^{\mu} \times \dot H^{\mu-1}))$ and for all $J \Subset I$ we have 
$\| u \|_{X(J)} <\infty$
for any Strichartz admissible norm $X(J)$ satisfying the scaling condition with $s = \mu$.  
\end{ppn}

\begin{ppn}[Long--time perturbation theory]\label{longtime}
Let $M, A, A' > 0$. There exists $\epsilon_0 = \epsilon_0(M,A,A') > 0$ and $\alpha > 0$ with the following property.  Let 
$\epsilon < \epsilon_0$, and let $(u_0,u_1) \in \energysp$.  Assume that 
$\tilde u$ is defined on $I \times \R^d$ and solves 
\begin{align*}
\partial_t^2 \tilde u - \Delta \tilde u + \mu |\tilde u|^{p-1} \tilde u = e, \quad \mbox{on } I \times \R^d,
\end{align*}
such that $\| |D_x|^{s_p - 1/2} \tilde u \|_{W(J)}  < + \infty$ for every $J \Subset I$, 
\begin{align*}
\sup_{t \in I} \| \vec{\tilde u}(t) \|_{\energysp} &\leq A, \\
\| \tilde u \|_{S(I)} &\leq M, \\
\| \vec{\tilde u}(0) - (u_0,u_1) \|_{\energysp} &\leq A', \\
\| |D_x|^{s_p - 1/2} e \|_{N(I)} + \| S(t)( \vec{\tilde u}(0) - (u_0,u_1)) \|_{S(I)} &\leq \epsilon.
\end{align*}
Then there exists a unique solution $u$ to \eqref{nlw} defined on $I$ with $\vec u(0) = (u_0,u_1)$, and $u$ satisfies
\begin{align*}
\| u \|_{S(I)} &\leq C(M,A,A'), \\
\sup_{t \in I} \| \vec{\tilde u}(t) - \vec u(t) \|_{\energysp} &\leq C(M,A,A')(A' + \epsilon + \epsilon^{\alpha}).
\end{align*}
\end{ppn}

We remark that the Strichartz estimates as previously stated do not quite suffice to prove Proposition \ref{longtime}.  One needs 
Strichartz estimates for the Duhamel operator
\begin{align*}
\left \| \int_0^t \frac{\sin((t-s)\sqrt{-\Delta})}{\sqrt{-\Delta}} F(s) ds \right \|_{L^q_tL^r_x} \lesssim \| F \|_{L^{a'}_tL^{b'}_x}
\end{align*}
that hold for a wider range of pairs (Corollary 8.7 of \cite{tag}).  Again, see for example
\cite{km11} for details.

From the local well--posedness theory and Strichartz estimates, one has the following 
criteria for blow--up and scattering:
\begin{description}
 \item[Blow--up criterion] if $T_+ < +\infty$, then for all $T \in I_{\max}(u)$ we have $\| u \|_{S(T,T_+)} = +\infty$, 
\item[Scattering criterion] there exists $T \in I_{\max}(u)$ such that  $\| u \|_{S(T,T_+)} < +\infty$ if and only if $T_+ = +\infty$ and 
$u$ scatters at $+\infty$, i.e. there exists a solution $v_L$ to the free wave equation $\partial_t^2 v_L - \Delta v_L = 0$ such that 
\begin{align*}
\lim_{t \rightarrow +\infty} \| \vec u(t) - \vec v_L(t) \|_{\energysp} = 0.
\end{align*}
\end{description}
Similar statements hold in the negative time direction.

By Strichartz estimates, the scattering criterion, and the local well--posedness theory we have the following 
small data theory:
\begin{description}
 \item[Small data theory] there exists $\bar \delta > 0$ such that
if $\| (u_0,u_1) \|_{\energysp} < \bar \delta$, then the unique solution $u$ to \eqref{nlw} with initial data $(u_0,u_1)$ is 
globally defined and scatters at $\pm \infty$.  Moreover, $u$ satisfies the estimate
\begin{align*}
\sup_{t \in \R} \| \vec u(t) \|_{\energysp} + \| u \|_{S(\R)} + \| |D_x|^{s_p - 1/2} u \|_{W(\R)} \leq C \| (u_0,u_1) \|_{\energysp}.
\end{align*}

\end{description}  

\subsection{Concentration compactness}

We now summarize the method of concentration compactness applied in our setting.  We use the notation from \cite{km10}. For 
$(u_0, u_1) \in \energysp$, we denote the solution to \eqref{nlw}  with initial data $(u_0,u_1)$ by $u$ which is defined 
on $I_{\max}(u) = (T_-,T_+)$. For $A > 0$, we define 
\begin{align*}
\mathcal{B}(A) := \left \{ (u_0,u_1) \in \energysp : \sup_{t \in I_{\max}(u)} \| \vec u(t) \|_{\energysp} \leq A  \right \}.
\end{align*}
We say that the property $\mathcal{SC}(A)$ holds if for all $(u_0,u_1) \in \mathcal{B}(A)$ we have 
$\| u \|_{S\left (I_{\max}(u)\right )} < +\infty$.  We say that $\mathcal{SC}(A; \vec u)$ holds if $\vec u(0) \in \mathcal{B}(A)$
and $\| u \|_{S\left (I_{\max}(u)\right )} < +\infty$. Recall from the comments following the local well--posedness theory, that
$\| u \|_{S\left (I_{\max}(u)\right )} < +\infty$ is equivalent to the statement that $u$ is globally defined and scatters. 

We now assume that Theorem \ref{thm1} is false.  By the small data theory, we have for 
all $A < \bar \delta$, $\mathcal{SC}(A)$ holds.  We then define
\begin{align*}
A_C := \sup \{ A > 0 : \mathcal{SC}(A) \mbox{ holds } \} > 0. 
\end{align*}
Our assumption that Theorem \ref{thm1} is false is equivalent to the statement that $A_C < +\infty$. 

In order to state the conclusion assuming Theorem \ref{thm1} is false, we need the following definition.

\begin{defn}
Let $I$ be a time interval and let $u$ be a solution to \eqref{nlw} on $I$.  We say that $u$ has the compactness property on $I$
if there exists a continuous function $N : I \rightarrow (0,+\infty)$ such that the set
\begin{align}
K = \left \{ \left ( \frac{1}{N(t)^{\frac{2}{p-1}}} u\left (t, \frac{\cdot}{N(t)} \right ) ,
\frac{1}{N(t)^{\frac{2}{p-1}+1}} \partial_t u\left (t, \frac{\cdot}{N(t)} \right )  \right ) : t \in I \right \}
\end{align}
is pre--compact in $\energysp$.
\end{defn}

We say that $u$ has the compactness property if $u$ has the compactness property on $I_{\max}(u)$.  We remark here
that by Arzela-Ascoli if $u$ has the compactness property on $I$, then for every $\eta > 0$, there exists a constant
$C = C(\eta)$ such that
\begin{align}
\int_{|x| \geq C(\eta)/N(t)} \left | (-\Delta)^{s_p/2} u(t) \right |^2 dx + 
\int_{|\xi| \geq C(\eta)N(t)} |\xi|^{2s_p}\left  |\hat u(t) \right |^2 d\xi < \eta, \\
\int_{|x| \geq C(\eta)/N(t)} \left | (-\Delta)^{(s_p-1)/2} \partial_t u(t) \right |^2 dx + 
\int_{|\xi| \geq C(\eta)N(t)} |\xi|^{2(s_p-1)} \left |\hat \partial_t u(t) \right |^2 d\xi < \eta 
\end{align}
for all $t \in I$.

Using Bahouri--Gerard profile decompositions (see \cite{fanelli} for $s_p > 1/2$ and \cite{ramos} for $s_p = 1/2$), the local well--posedness theory, and the long time 
perturbation theory, one reaches the following conclusion assuming Theorem \ref{thm1} is false.  For complete details see 
for example \cite{km10} \cite{shen} \cite{kill0} \cite{murphy} \cite{tao} \cite{tao2}. 

\begin{ppn}\label{contra}
Suppose that Theorem \ref{thm1} is false.  Then there exists a solution $u$ to \eqref{nlw} with the compactness property such that 
$\mathcal{SC}(A_C, \vec u)$ fails.  Moreover, we may assume that $I_{\max}(u) = (T_-,+\infty)$ and 
that one of the following cases holds: 
\begin{itemize}
\item $\liminf_{t \rightarrow +\infty} N(t) = 0$ (frequency cascade case),
\item $N(t) \equiv 1$ for all $t$ and $T_- = -\infty$ (soliton--like case). 
\end{itemize}
\end{ppn}

In the remainder of this work, we show that if $u$ is as in Proposition \ref{contra}, then $\vec u = 0$, a contradiction.  Thus,
our main result, Theorem \ref{thm1}, is true. 
 
Although the notion of solutions with the compactness property may at first seem like a construct specific to the 
concentration compactness approach, it has been recently proven that this is not necessarily the case.  In particular, in \cite{dkm6} it was
shown that for any nonlinear dispersive equation with a strong enough local well--posedness theory and a notion of a profile decomposition, any type II solution 
that blows--up in finite time or does not scatter must converge weakly to a nonzero solution with the compactness property.  So, in a 
sense, solutions with the compactness property arise naturally when studying the long time asymptotics of type II solutions to 
dispersive equations.  

A standard fact about solutions to \eqref{nlw} with the compactness property is that the linear part converges weakly to 0
as $t \rightarrow T_-$ or $t \rightarrow T_+$.  In fact this is deduced from the stronger statement that any Strichartz norm 
of the linear part vanishes asymptotically as $t \rightarrow T_-$ or $t \rightarrow T_+$ (see \cite{tao2} Section 6 or \cite{shen} Proposition
3.6).  The following lemma will be crucial in establishing higher regularity for solutions to \eqref{nlw} with the compactness property. 

\begin{lem}\label{lem1}
Let $u$ be a solution to \eqref{nlw} on $I_{\max}(u) = (T_-,+\infty)$ with the compactness property.  Then for any $t_0 \in I$ we have
\begin{align}
&\int_{t_0}^T S(t_0 - t)(0, \mu |u|^{p-1}u) dt \rightharpoonup \vec u(t_0) \quad \mbox{as } T \rightarrow +\infty \mbox{ in }
\energysp \\
-&\int^{t_0}_T S(t_0 - t)(0, \mu |u|^{p-1}u) dt \rightharpoonup \vec u(t_0) \quad \mbox{as } T \rightarrow T_- \mbox{ in }
\energysp.
\end{align}
\end{lem}

Using Lemma \ref{lem1}, we will show in the following section that a solution $u$ as in Proposition \ref{contra} satisfies
\begin{align*}
\| \vec u(t) \|_{\dot H^1 \times L^2} \lesssim N(t)^{1-s_p}. 
\end{align*}
Since we are in the subcritical regime $s_p < 1$, we have $E(\vec u) \leq 0$ if $u$ is in the frequency cascade case. 
To show that $\vec u = 0$ in the frequency cascade case, we will need the following blow--up result for solutions with nonpositive energy (see \cite{kill} Theorem 3.1). 

\begin{ppn}\label{prop3}
Assume $1/2 \leq s_p < 1$.  Let $\vec u(t) \in C\left (I_{\max}(u); (\dot H^1 \times L^2) 
\cap (\energysp)\right )$ be a solution to \eqref{nlw} on $I_{\max}(u) = (T_-,T_+)$ in the focusing case.  If $E(\vec u(t)) \leq 0$ then either
$I_{\max}(u)$ is a finite interval or $u \equiv 0$. 
\end{ppn}

The proof of Proposition \ref{prop3} uses the formal virial 
identity
\begin{align*}
\frac{d^2}{dt^2} \int |u(t)|^2 dx = 2 \int |\partial_t u(t)|^2 dx - 2 \int |\nabla u(t)|^2 dx + 2 \int |u(t)|^{p+1} dx. 
\end{align*}
and convexity arguments based on the works \cite{levine} \cite{glass}.  However, since $u(t)$ is not necessarily in $L^2$, one must truncate and show that all errors 
are suitably small in order to close the argument.  It is in showing the errors are small that the subcritical assumption 
$s_p < 1$ becomes crucial.  

\section{Higher Regularity for Compact Solutions}

In this section, we show that a solution $u$ to \eqref{nlw} on $I_{\max}(u) = (T_-,+\infty)$ with the compactness property has more
regularity than $\energysp$.  This is achieved in two steps.  The first step is contained in the following proposition.

\begin{ppn}\label{prop1}
Let $s_0 \in (s_p, 1)$ be defined by the relation
$$
s_0 = \frac{d}{2} - \frac{d+2}{2p}.
$$
Let $u$ be a solution to \eqref{nlw} on $(T_-,+\infty)$ with the compactness property.  Then for every $t \in (T_-, +\infty)$,
$\vec u(t) \in \dot H^{s_0} \times \dot H^{s_0 - 1}$ and 
$$
\| \vec u(t) \|_{\dot H^{s_0} \times \dot H^{s_0 - 1}} \lesssim N(t)^{s_0 - s_p}.
$$
\end{ppn}

Using Proposition \ref{prop1}, we then show that $\vec u(t)$ is in the energy space $\dot H^1 \times L^2$ with a size estimate in time.

\begin{ppn}\label{prop2}
Let $u$ be a solution to \eqref{nlw} on $(T_-, +\infty)$ with the compactness property.  Then for every $t \in (T_-, +\infty)$,
$\vec u(t) \in \dot H^1 \times L^2$ and 
\begin{align}
\| \vec u(t) \|_{\dot H^1 \times L^2} \lesssim N(t)^{1-s_p}.
\end{align}
\end{ppn}

We remark that the implicit constants in the previous two propositions depend on $d$, $p$, and
\begin{align*} 
\sup_{t \in (T_-,+\infty)}
\| \vec u(t)\|_{\energysp}
\end{align*}
but not $t$. 

The proofs of these two propositions draw from the proofs of the corresponding statements 
for the $3d$ cubic equation found in Section 3 of \cite{dodl1} but contain refinements since 
we are no longer in such a specific setting.  We first prove Proposition \ref{prop2} using Proposition \ref{prop1}.

\subsection{Proof of Proposition \ref{prop2} using Proposition \ref{prop1}}  We first require the following lemma.

\begin{lem}\label{lem3}
Let $u$ satisfy the conclusion of Proposition \ref{prop1}.  Then there exists a small $\delta > 0$ and a constant
$C > 0$ such that for every $t_0 \in
I_{\max}(u)$
\begin{align}
\| u \|_{L^p_tL^{2p}_x([t_0 - \delta / N(t_0), t_0 + \delta / N(t_0)] \times \R^d)} \leq C N(t_0)^{s_0 - s_p}.
\end{align}
\end{lem}

\begin{proof}
We first note (see for example Corollary 3.6 of \cite{kill0}) that there exists $\tilde \delta > 0$ such that 
for all $t_0 \in I_{\max}(u)$, $[t_0 - \tilde \delta /N(t_0), t_0 + \tilde \delta/N(t_0)] \subset
I_{\max}(u)$. Let $\delta < \tilde \delta$ to be chosen later.  Define $J = [t_0 - \delta / N(t_0), t_0 + \delta / N(t_0)]$.  For clarity in the exposition, we consider the $4d$ and
$5d$ cases separately.  In $4d$, we have $s_0 = 2 - \frac{3}{p}$, $s_0 - s_p = \frac{3-p}{p(p-1)}$, and by Sobolev embedding
$$
\dot H^{s_0} \hookrightarrow L^{\frac{4p}{3}}.
$$
Let
$$
X(J) = L^p_tL^{2p}_x(J \times \R^4) \cap L^{\infty}_t
L_x^{\frac{4p}{3}}(J \times \R^4).
$$
Define a Strichartz admissible pair $(a,b)$ by
\begin{align*}
a' &= \frac{5p}{8p - 9}, \\
b' &= \frac{10p}{p + 12},
\end{align*}
and note that
\begin{align*}
\frac{1}{pb'} = (1- \theta) \frac{1}{2p} + \theta \frac{3}{4p} \quad \mbox{with} \quad
\theta = \frac{24 - 8p}{5p} \in (0,1).
\end{align*}
By Strichartz estimates and H\"older's inequality in time, we have that
\begin{align*}
\| u \|_{X(J)} &\lesssim \| \vec u(t_0) \|_{\dot H^{s_0} \times \dot H^{s_0-1}} +
\| |u|^p \|_{L^{a'}_t L^{b'}_x(J \times \R^4)} \\
&\lesssim N(t_0)^{s_0 - s_p} +
\left \|
\| u \|_{L^{2p}_x}^{p(1-\theta)} \| u \|^{p\theta}_{L^{4/3p}_x}
\right \|_{L^{a'}_t(J)} \\
&\lesssim N(t_0)^{s_0 - s_p} + \| u \|_{X(J)}^p \delta^{\frac{3-p}{p}} N(t_0)^{-\frac{3-p}{p}} \\
&\lesssim N(t_0)^{s_0 - s_p} + \delta^{\frac{3-p}{p}} \left ( \| u \|_{X(J)} N(t_0)^{-(s_0 - s_p)} \right )^{p-1}
\| u \|_{X(J)}.
\end{align*}
By choosing $\delta$ sufficiently small, a standard continuity argument finishes the proof in $4d$.  In $5d$,
we have $s_0 = \frac{5}{2} - \frac{7}{2p}$ and we define
$$
X(J) =  L^p_tL^{2p}_x(J \times \R^4) \cap L^{\infty}_t
L_x^{\frac{10p}{7}}(J \times \R^4).
$$
We choose a Strichartz pair $(a,b)$ by
\begin{align*}
a' &= \frac{3p}{6p - 7}, \\
b' &= \frac{6p}{7},
\end{align*}
and note that
\begin{align*}
\frac{1}{pb'} = (1-\theta) \frac{1}{2p} + \theta \frac{7}{10p} \quad \mbox{with} \quad
\theta = 5\frac{7-3p}{6p} \in (0,1).
\end{align*}
The proof proceeds as in the $4d$ case, and we omit the details.
\end{proof}

The previous lemma immediately implies that there exists a small $\delta > 0$ such that
for every $t_0 \in I_{\max}(u)$
\begin{align}
\| |u|^{p-1}u \|_{L^1_tL^{2}_x([t_0 - \delta / N(t_0), t_0 + \delta / N(t_0)] \times \R^d)} \lesssim N(t_0)^{p(s_0 - s_p)}
= N(t_0)^{1-s_p}. \label{lembd}
\end{align}

\begin{proof}[Proof of Proposition \ref{prop2}]
Let $t_0 \in I_{\max}(u)$, and fix $\delta$ from Lemma \ref{lem3}.  By time translation we may assume that $t_0 = 0$.  Define
\begin{align}
v = u + \frac{i}{\sqrt{-\Delta}} \partial_t u,
\end{align}
so that
$$
\| \vec u(t) \|_{\dot H^1 \times L^2} \simeq \| v \|_{\dot H^1}.
$$
Since $u$ solves \eqref{nlw}, $v$ solves the equation
\begin{align}
\partial_t v = -i \sqrt{-\Delta} v - \mu \frac{i}{\sqrt{-\Delta}} |u|^{p-1} u.
\end{align}
By Duhamel's principle, we have for any $T$ with $T_- < T < 0$
\begin{align}
 v(0) = e^{iT\sqrt{-\Delta}} v(T) - \mu \frac{i}{\sqrt{-\Delta}} \int^0_T e^{i\tau \sqrt{-\Delta}} |u(\tau)|^{p-1}u(\tau)
d \tau.
\end{align}
We define an approximate identity as follows.  Let $\psi \in C^\infty_0(\R^d)$ be radial such that $\| \psi \|_{L^1}
= 1$.  For $M > 0$, define $\psi_M = M^d \psi(M \cdot)$ and
\begin{align*}
Q_Mf = \psi_M \ast f.
\end{align*}
To prove Proposition \ref{prop2}, it suffices to show that for all $M \geq M_0$
$$
\| Q_M v(0) \|_{\dot H^1} \lesssim N(0)^{1-s_p}
$$
where the implied constant is independent of $M$ and $N(0)$.

A key tool in the proof is Lemma \ref{lem1}.  Let $T_- < T_1 < 0 < T_2 < T_+$ and write (by the Duhamel formula)
\begin{align*}
v(0) &= e^{iT_2 \sqrt{-\Delta}} v(T_2) + \mu \frac{i}{\sqrt{-\Delta}} \int_0^{T_2} e^{it \sqrt{-\Delta}} |u(t)|^{p-1}u(t)
d t \\
&= e^{iT_1\sqrt{-\Delta}} v(T_1) - \mu \frac{i}{\sqrt{-\Delta}} \int^0_{T_1} e^{i\tau \sqrt{-\Delta}} |u(\tau)|^{p-1}u(\tau)
d \tau.
\end{align*}
Denote the $L^2$ and $\dot H^1$ pairings by
\begin{align*}
\la f, g \ra_{L^2} &= \Re \int f \overline{g} dx, \\
\la f, g \ra_{\dot H^1} &= \la \sqrt{-\Delta} f, \sqrt{-\Delta} g \ra_{L^2}.
\end{align*}
Then we may write
\begin{align*}
\left \la Q_M v(0) , Q_M v(0) \right \ra_{\dot H^1}
= \left \la Q_M \left ( \sqrt{-\Delta}  e^{iT_2 \sqrt{-\Delta}} v(T_2) + i \mu \int_0^{T_2} e^{it \sqrt{-\Delta}} |u(t)|^{p-1}u(t) dt
\right ), \right. \\
\left. Q_M \left ( \sqrt{-\Delta}  e^{iT_1 \sqrt{-\Delta}} v(T_1) - i \mu \int^0_{T_1} e^{i\tau \sqrt{-\Delta}} |u(\tau)|^{p-1}u(\tau) d\tau
\right )  \right \ra_{L^2}.
\end{align*}

We first focus on estimating the term containing the two Duhamel integrals
\begin{align}
\left \la
Q_M \left (
\int_0^{T_2} e^{it \sqrt{-\Delta}} |u(t)|^{p-1}u(t) dt,
\right ),
Q_M \left (
\int^0_{T_1} e^{i\tau \sqrt{-\Delta}} |u(\tau)|^{p-1}u(\tau) d\tau,
\right )
\right \ra_{L^2}. \label{prop21}
\end{align}
We split the Duhamel integrals into two pieces, one of which we have good control (uniformly in $T_1,T_2$) of the nonlinearity by Lemma \ref{lem3} and radial Sobolev embedding. To do this,
we introduce a smooth cutoff in space.  Let $\varphi \in C^{\infty}_0(\R^d)$ be radial with $\varphi(x) = 1$ for $|x| \leq \frac{1}{8}$ and
$\varphi(x) = 0$ for $|x| \geq \frac{1}{4}$.  Define
\begin{align}
A &:= Q_M \left ( \int_0^{\delta/N(0)} e^{it \sqrt{-\Delta}} |u(t)|^{p-1}u(t) dt +
\int_{\delta/N(0)}^{T_2} e^{it \sqrt{-\Delta}}\left ( 1 - \varphi \left ( \frac{x}{t} \right) \right )
 |u(t)|^{p-1}u(t) dt \right ), \\
B &:= Q_M \left (\int_{\delta/N(0)}^{T_2} e^{it \sqrt{-\Delta}} \varphi \left ( \frac{x}{t} \right ) |u(t)|^{p-1}u(t) dt \right ),
\end{align}
so that
\begin{align*}
Q_M \left (\int_0^{T_2} e^{it \sqrt{-\Delta}} |u(t)|^{p-1}u(t) dt \right ) = A + B.
\end{align*}
We define $A'$ and $B'$ to be the corresponding integrals in the negative time direction so that
\begin{align*}
Q_M \left (
\int^0_{T_1} e^{i\tau \sqrt{-\Delta}} |u(\tau)|^{p-1}u(\tau) d\tau,
\right ) = A' + B'.
\end{align*}

We now estimate $A$.  By Lemma \ref{lem3}, we have that
\begin{align*}
\left \|
Q_M \left ( \int_0^{\delta/N(0)} e^{it \sqrt{-\Delta}} |u(t)|^{p-1}u(t) dt \right )
\right \|_{L^2}
\lesssim \int_0^{\delta/N(0)}  \| u \|_{L^{2p}_x}^p dt \lesssim N(0)^{1-s_p}
\end{align*}
For the second term appearing in $A$, we have
\begin{align}
\left \| Q_M  \left ( \int_{\delta/N(0)}^{T_2} e^{it \sqrt{-\Delta}}\left ( 1 - \varphi \left ( \frac{x}{t} \right) \right )
 |u(t)|^{p-1}u(t) dt \right ) \right \|_{L^2} \\
\lesssim \int_{\delta/N(0)}^{T_2} \left \| \left ( 1 - \varphi \left ( \frac{x}{t} \right) \right )
 |u(t)|^{p-1}u(t) \right \|_{L^2_x} dt
\end{align}
By radial Sobolev embedding with
\begin{align*}
\beta = \frac{2-s_p}{p}, \quad q = 2p, \quad s = s_p,
\end{align*}
we have that
\begin{align}
\left \| \left ( 1 - \varphi \left ( \frac{x}{t} \right) \right )
 |u(t)|^{p-1}u(t) \right \|_{L^2_x} \lesssim |t|^{s_p - 2} \left \| |x|^{\frac{2-s_p}{p}} u(t) \right \|^p_{L^{2p}_x}
\lesssim |t|^{s_p - 2} \| u(t) \|^p_{\dot H^{s_p}} \lesssim |t|^{s_p - 2}. \label{prop42}
\end{align}
Hence
\begin{align}
\int_{\delta/N(0)}^{T_2} \left \| \left ( 1 - \varphi \left ( \frac{x}{t} \right) \right )
 |u(t)|^{p-1}u(t) \right \|_{L^2_x} dt \lesssim N(0)^{1-s_p} \label{prop25}
\end{align}
uniformly in $T_2$.  Thus,
\begin{align}
\| A \|_{L^2} \lesssim N(0)^{1-s_p}. \label{prop22}
\end{align}
Similarly
\begin{align}
\| A' \|_{L^2} \lesssim N(0)^{1-s_p} \label{prop23}
\end{align}
uniformly in $T_1$.

We write \eqref{prop21} as a sume of pairings
$$
\la A + B , A' + B' \ra = \la A + B, A' \ra + \la A, A' + B' \ra + \la B , B' \ra - \la A, A' \ra.
$$
We estimate $\la A, A' \ra$ using \eqref{prop22} and \eqref{prop23},
\begin{align}
| \la A , A' \ra | \leq \| A \|_{L^2} \| A' \|_{L^2} \lesssim N(0)^{2(1-s_p)}
\end{align}

We claim that
\begin{align}
| \la B, B' \ra | \lesssim M^{-1},
\end{align}
where the implied constant depends on $\| u \|_{L^{\infty}_t (\energysp)}$, $\delta$, and $N(0)$ but not on $T_1$, $T_2$, nor $M$. To show this,
we use that $\la e^{it\sqrt{-\Delta}} f , g\ra = \la f , e^{-it\sqrt{-\Delta}} g \ra$ to write
\begin{align*}
\la B , B' \ra = \int_{\delta/N(0)}^{T_2} \int^{-\delta/N(0)}_{T_1}
\left \la \varphi \left ( \frac{x}{|t|} \right ) |u(t)|^{p-1}u(t), Q_M^* Q_M e^{i(\tau - t)\sqrt{-\Delta}}
\varphi \left ( \frac{y}{|\tau|} \right ) |u(\tau)|^{p-1}u(\tau) \right \ra_{L^2} d\tau dt
\end{align*}
Using the inverse Fourier transform, we note that $Q_M^* Q_M e^{i(\tau - t)\sqrt{-\Delta}}$ is given by convolution with a kernel
\begin{align}
K(x) = K(|x|) = c(d) \int_0^{\pi} \int_0^{+\infty} e^{i |x| \rho \cos \theta} e^{i(\tau - t)\rho} \left |
\hat \psi \left ( \frac{\rho}{M} \right )
\right |^2 \rho^{d-1} d\rho \sin^{d-1}\theta d\theta,
\end{align}
where $\hat \psi \in \mathcal S(\R^d)$ is the Fourier transform of $\psi$, and the integrand is written in polar coordinates on
$\R^d$ with $|\xi| = \rho$.  Since
\begin{align*}
\left (1 - \frac{d^2}{d\rho^2} \right ) e^{i\rho M (\cos \theta |x| - (t-\tau))} =
\left (1 + M^2 |\cos \theta |x| - (t-\tau)|^2 \right ) e^{i\rho M (\cos \theta |x| - (t-\tau))},
\end{align*}
we have after integrating by parts $L$ times in $\rho$ the estimate
\begin{align}
|K(x-y)| \lesssim_L \int_0^{\pi} \frac{M^{d}}{\la M |(t - \tau) - \cos \theta |x-y|| \ra^L} d\theta. \label{prop24}
\end{align}

In the spatial $L^2$ pairing appearing in the above expression for $\la B, B' \ra$, we have the constraints $\tau < -\delta/N(0) < \delta/N(0) < t$, $|x| \leq \frac{|t|}{4}$, $|y| \leq \frac{|\tau|}{4}$,
so that $|x-y| \leq \frac{1}{2}(|t| + |\tau|) = \frac{1}{2}(t-\tau)$.  Thus, we have by \eqref{prop24}
\begin{align}
|K(x-y)| \lesssim_L \frac{M^{d}}{\la M (|t| + |\tau|) \ra^L},
\end{align}
whence by H\"older's inequality in space and the Sobolev embedding $\dot H^{s_p} \hookrightarrow L^{\frac{d(p-1)}{2}}$ we have
\begin{align}
\sup_{|x| \leq |t|/4} \left | K \ast \left (\varphi \left ( \frac{y}{|\tau|} \right ) |u(\tau)|^{p-1}u(\tau)  \right )(x) \right |
&\lesssim \frac{M^{d}}{\la M (|t| + |\tau| \ra) \ra^L} \int_{|y| \leq |\tau|/4} |u(\tau)|^p dy \\
&\lesssim \frac{M^{d}|\tau|^{d-\frac{2p}{p-1}}}{\la M (|t| + |\tau|) \ra^L}
\end{align}
where the implied constant does not depend on $M$.  Thus, by choosing $L$ sufficiently large we have
\begin{align*}
\la B , B' \ra
&\lesssim \int_{\delta/N(0)}^{+\infty} \int^{-\delta/N(0)}_{-\infty} \frac{M^{d}
|t|^{d-\frac{2p}{p-1}} |\tau|^{d-\frac{2p}{p-1}}}{\la M (|t| + |\tau|) \ra^L} d\tau dt \\
&\lesssim M^{-1}
\end{align*}
where the implied constant depends on $\| u \|_{L^{\infty}_t (\energysp)}$, $\delta$, and $N(0)$.

We now turn to estimating the remaining terms $\la A , A' + B' \ra$ and $\la A + B , A' \ra$.  We will focus on the term
$\la A , A' + B' \ra$ since estimating $\la A + B, A' \ra$ is similar. We recall that $\la A , A'+B' \ra$ is given by
\begin{align*}
\left \la
Q_M \left ( \int_0^{\delta/N(0)} e^{it \sqrt{-\Delta}} |u(t)|^{p-1}u(t) dt +
\int_{\delta/N(0)}^{T_2} e^{it \sqrt{-\Delta}}\left ( 1 - \varphi \left ( \frac{x}{t} \right) \right )
 |u(t)|^{p-1}u(t) dt \right ) , \right.  \\
 \left. Q_M \left ( \int^0_{T_1} e^{i\tau \sqrt{-\Delta}} |u(\tau)|^{p-1} u(\tau) d\tau \right )
\right \ra_{L^2}
\end{align*}
By the Duhamel formula, we write
\begin{align*}
 Q_M \left ( \int^0_{T_1} e^{i\tau \sqrt{-\Delta}} |u(\tau)|^{p-1} u(\tau) d\tau \right )
= \mu i \sqrt{-\Delta} Q_M v(0) - \mu i\sqrt{-\Delta} e^{i T_1 \sqrt{-\Delta}} Q_M v(T_1).
\end{align*}
By \eqref{prop22} we have
\begin{align*}
\left \la
Q_M \left ( \int_0^{\delta/N(0)} e^{it \sqrt{-\Delta}} |u(t)|^{p-1}u(t) dt +
\int_{\delta/N(0)}^{T_2} e^{it \sqrt{-\Delta}}\left ( 1 - \varphi \left ( \frac{x}{t} \right) \right )
 |u(t)|^{p-1}u(t) dt \right ) , \right.  \\
 \left. \sqrt{-\Delta} Q_M v(0)
\right \ra_{L^2} \lesssim N(0)^{1-s_p} \| Q_M v(0) \|_{\dot H^1}.
\end{align*}

To handle the second half of the pairing, we use Lemma \ref{lem1} and let $T_1 \rightarrow T_-$ and $T_2 \rightarrow T_+$ so it is important to
reemphasize that all previous estimates are uniform in $T_1$ and $T_2$. We first note that by
\eqref{prop25} and the fact that $\hat \psi \in \mathcal{S}(\R^d)$, we have (for fixed $M$)
\begin{align*}
(-\Delta)^{\frac{1-s_p}{2}} Q_M \left ( \int_0^{\delta/N(0)} e^{it \sqrt{-\Delta}} |u(t)|^{p-1}u(t) dt +
\int_{\delta/N(0)}^{T_2} e^{it \sqrt{-\Delta}}\left ( 1 - \varphi \left ( \frac{x}{t} \right) \right )
 |u(t)|^{p-1}u(t) dt \right )
\end{align*}
converges to
\begin{align*}
(-\Delta)^{\frac{1-s_p}{2}} Q_M \left ( \int_0^{\delta/N(0)} e^{it \sqrt{-\Delta}} |u(t)|^{p-1}u(t) dt +
\int_{\delta/N(0)}^{T_+} e^{it \sqrt{-\Delta}}\left ( 1 - \varphi \left ( \frac{x}{t} \right) \right )
 |u(t)|^{p-1}u(t) dt \right )
\end{align*}
in $L^2$ as $T_2 \rightarrow T_+$.  By Lemma \ref{lem1}
\begin{align*}
(-\Delta )^{\frac{s_p}{2}} e^{i T_1 \sqrt{-\Delta}} Q_M v(T_1) \rightharpoonup 0 \quad \mbox{in } L^2 \mbox{ as }
T_1 \rightarrow T_-.
\end{align*}
Thus, for fixed $M$
\begin{align*}
\lim_{T_1 \rightarrow T_-} \lim_{T_2 \rightarrow T_+} &
\left \la
\sqrt{-\Delta} e^{i T_1 \sqrt{-\Delta}} Q_M v(T_1), \right. \\ & \left.
 Q_M \left ( \int_0^{\delta/N(0)} e^{it \sqrt{-\Delta}} |u(t)|^{p-1}u(t) dt +
\int_{\delta/N(0)}^{T_2} e^{it \sqrt{-\Delta}}\left ( 1 - \varphi \left ( \frac{x}{t} \right) \right )
 |u(t)|^{p-1}u(t) dt \right )
\right \ra \\
&= 0.
\end{align*}
Hence we have proved that
\begin{align*}
\left |\lim_{T_1 \rightarrow T_-} \lim_{T_2 \rightarrow T_+}
\la A , A' + B' \ra \right | \lesssim N(0)^{1-s_p} \| Q_M v(0) \|_{\dot H^1}.
\end{align*}
Similarly it follows that
\begin{align*}
\left | \lim_{T_1 \rightarrow T_-} \lim_{T_2 \rightarrow T_+}
\la A + B, A' \ra \right | \lesssim N(0)^{1-s_p} \| Q_M v(0) \|_{\dot H^1}.
\end{align*}
In conclusion, we have proved that for all $M$ sufficiently large
\begin{align*}
\left |
\lim_{T_1 \rightarrow T_-} \lim_{T_2 \rightarrow T_+}
\int_0^{T_2} \int^0_{T_1} \left \la e^{it \sqrt{-\Delta}} Q_M |u(t)|^{p-1}u(t) dt,
 e^{i\tau \sqrt{-\Delta}} Q_M |u(\tau)|^{p-1}u(\tau) \right \ra d\tau dt
\right | \\ \lesssim N(0)^{2(1-s_p)} + N(0)^{1-s_p}\| Q_M v(0) \|_{\dot H^1}.
\end{align*}

What remains to estimate in $\la Q_M v(0) , Q_M v(0) \ra_{\dot H^1}$ are the terms that contain at most one
Duhamel integral.  We first consider the term
\begin{align*}
\left \la \sqrt{-\Delta} \la e^{iT_2 \sqrt{-\Delta}} Q_M v(T_2) , \sqrt{-\Delta} e^{iT_1 \sqrt{-\Delta}} Q_M v(T_1) \right \ra_{L^2}.
\end{align*}
For fixed $M$ and $T_1 > T_-$, we have that $(-\Delta)^{\frac{2-s_p}{2}} e^{iT_1 \sqrt{-\Delta}} Q_M v(T_1) \in L^2$.  By
Lemma \ref{lem1} we conclude that
\begin{align*}
\lim_{T_2 \rightarrow T_+} \left \la
\sqrt{-\Delta}  e^{iT_2 \sqrt{-\Delta}} Q_M v(T_2) , \sqrt{-\Delta}  e^{iT_1 \sqrt{-\Delta}} Q_M v(T_1)
\right \ra_{L^2}
= 0.
\end{align*}

To estimate the term
\begin{align*}
\left \la
\sqrt{-\Delta} e^{iT_2 \sqrt{-\Delta}} Q_M v(T_2), Q_M \left ( \int^0_{T_1} e^{i \tau \sqrt{-\Delta}} |u(\tau)|^{p-1} u(\tau) d\tau
\right )
\right \ra_{L^2},
\end{align*}
we first note that by \eqref{lembd} $\int^0_{T_1} e^{i \tau \sqrt{-\Delta}} |u(\tau)|^{p-1} u(\tau) d\tau  \in L^2$.  Thus
\begin{align*}
(-\Delta)^{\frac{1-s_p}{2}} Q_M \left ( \int^0_{T_1} e^{i \tau \sqrt{-\Delta}} |u(\tau)|^{p-1} u(\tau) d\tau \right )
\in L^2,
\end{align*}
so by Lemma \eqref{lem1}
\begin{align*}
\lim_{T_2 \rightarrow T_+} \left \la
\sqrt{-\Delta} e^{iT_2 \sqrt{-\Delta}} Q_M v(T_2), Q_M \left ( \int^0_{T_1} e^{i \tau \sqrt{-\Delta}} |u(\tau)|^{p-1} u(\tau) d\tau
\right )
\right \ra_{L^2} = 0.
\end{align*}

We now turn to our final term and claim that
\begin{align*}
\lim_{T_1 \rightarrow T_-} \lim_{T_2 \rightarrow T_+}
\left \la
\sqrt{-\Delta} e^{iT_1 \sqrt{-\Delta}} Q_M v(T_1), Q_M \left ( \int_0^{T_2} e^{i t \sqrt{-\Delta}} |u(t)|^{p-1} u(t) dt
\right )
\right \ra_{L^2} = 0.
\end{align*}
By the Duhamel formula
\begin{align*}
Q_M \int_0^{T_2} e^{i t \sqrt{-\Delta}} |u(t)|^{p-1} u(t) dt = -i\mu \sqrt{-\Delta} Q_M \left (v(0) - e^{iT_2 \sqrt{-\Delta}} v(T_2) \right).
\end{align*}
Again by Lemma \ref{lem1} and our previous computations we have 
\begin{align*}
 \lim_{T_1 \rightarrow T_-} \lim_{T_2 \rightarrow T_+}
\left \la
\sqrt{-\Delta} e^{iT_1 \sqrt{-\Delta}} Q_M v(T_1), Q_M \left ( \int_0^{T_2} e^{i t \sqrt{-\Delta}} |u(t)|^{p-1} u(t) dt
\right )
\right \ra_{L^2} \\ 
= \lim_{T_1 \rightarrow T_-} \left \la
\sqrt{-\Delta} e^{iT_1 \sqrt{-\Delta}} Q_M v(T_1), \sqrt{-\Delta} Q_M  v(0) 
\right \ra_{L^2} = 0.
\end{align*}

Thus, we have proved that for all $M \geq M_0$ 
\begin{align*}
\| Q_M v(0) \|_{\dot H^1}^2 \leq C_0 \| Q_M v(0) \|_{\dot H^1} N(0)^{1-s_p} + C_0 N(0)^{2(1-s_p)} + C_1 M^{-1},
\end{align*}
where the constant $C_0$ is independent of $M$ and $N(0)$ and the constant $C_1$ is independent of $M$. This shows that 
\begin{align*}
\limsup_{M \rar \infty} \| Q_M v(0) \|_{\dot H^1} \lesssim N(0)^{1-s_p}, 
\end{align*}
where the implied constant is independent of $N(0)$.  Thus, we see that $v(0) \in \dot H^1$ and $\| v(0) \|_{\dot H^1} \lesssim N(0)^{1-s_p}$ as desired. 
\end{proof}

\subsection{Proof of Proposition \ref{prop1}}  

We first remark that we will prove something stronger than 
Proposition \ref{prop1}.  In fact, we will prove the following.

\begin{ppn}\label{prop5}
Let $u$ be a solution to \eqref{nlw} on $(T_-,\infty)$ with the compactness 
property. Then for all $s \in [s_p,1)$ and $t \in (T_-,\infty)$, $\vec u(t) \in \dot H^s \times
\dot H^{s-1}$
and 
\begin{align}
\| \vec u(t)\|_{\dot H^s \times \dot H^{s-1}} \lesssim N(t)^{s - s_p}.
\end{align}
\end{ppn}

To start the proof of Proposition \ref{prop1}, we establish the following 
lemma. 

\begin{lem}\label{lem2}
Let $\eta > 0$.  There exists $\delta > 0$ such that for all $t_0 \in I$ 
\begin{align}
\| u \|_{S([t_0 - \delta / N(t_0), t_0 + \delta/N(t_0)]} < \eta.
\end{align}
\end{lem}

\begin{proof}
Assume without loss of generality that $t_0 = 0$.  Define $J = [-\delta/ N(0), \delta/N(0)]$,
and 
\begin{align*}
\| u \|_{X(J)} := \| u \|_{S([t_0 - \delta / N(t_0), t_0 + \delta/N(t_0)]}
+ \| |D_x|^{s_p - 1/2} u \|_{W([t_0 - \delta / N(t_0), t_0 + \delta / N(t_0)])}.
\end{align*}
Note that if $s_p = 1/2$, $X(J) = S(J)$.  By the Duhamel formula 
\begin{align}
\| u \|_{X(J)} \leq \left \| S(t)(u_0,u_1) \right \|_{X(J)} + \left \| 
\int_0^t \frac{\sin((t - \tau)\sqrt{-\Delta})}{\sqrt{-\Delta}} |u(\tau)|^{p-1} u(\tau) d\tau 
\right \|_{X(J)}.
\end{align}
We first estimate $\left \| S(t)(u_0,u_1) \right \|_{S(J)}$.  By compactness of the trajectory, there exists $C(\eta) > 0$ such 
that 
$$
\| P_{\geq C(\eta)N(t)} \vec u(t) \|_{\energysp} < \eta.
$$
The constant $C(\eta)$ is independent of $t$ which justifies us taking $t_0 = 0$ initially. By Strichartz estimates
\begin{align}
\| S(t) P_{\geq C(\eta)N(0)} \vec u(0) \|_{X(J)} \lesssim \| P_{\geq C(\eta)N(t)} \vec u(t) \|_{\energysp} \lesssim \eta 
. \label{ppn11}
\end{align}
We now estimate $\| S(t) P_{\leq C(\eta)N(0)} \vec u(0) \|_{X(J)}$.  By Sobolev embedding and Bernstein's inequalities, 
\begin{align*}
\| P_{\leq C(\eta)N(0)} S(t) \vec u(0) \|_{L_x^{(p-1)(d+1)/2}} &\lesssim
\| P_{\leq C(\eta) N(0)} S(t) \vec u(0) \|_{\dot H^{\frac{d}{2} - \frac{2d}{(p-1)(d+1)}}} \\
&\lesssim 
(C(\eta) N(0))^{\frac{d}{2} - \frac{2d}{(p-1)(d+1)} - s_p} \| S(t) u(0) \|_{\energysp} \\
&\lesssim C(\eta)^{\frac{2}{(p-1)(d+1)}} N(0)^{\frac{2}{(p-1)(d+1)}}.
\end{align*}
Taking the $L^{(p-1)(d+1)/2}_t(J)$ norm of both sides of the previous estimate yields 
\begin{align}
 \| P_{\leq C(\eta)N(0)} S(t) \vec u(0) \|_{S(J)} \lesssim C(\eta)^{\frac{2}{(p-1)(d+1)}} \delta^{\frac{2}{(p-1)(d+1)}}. 
\label{ppn12}
\end{align}
Similarly, by Sobolev embedding and Bernstein's inequalities,
\begin{align*}
\| |D_x|^{s_p - 1/2} P_{\leq C(\eta)N(0)} S(t) \vec u(0) \|_{L_x^{2(d+1)/(d-1)}} &\lesssim
\| P_{\leq C(\eta) N(0)} S(t) \vec u(0) \|_{\dot H^{\frac{d-1}{2(d+1)} + s_p}} \\
&\lesssim 
(C(\eta) N(0))^{\frac{d-1}{2(d+1)}} \| S(t) u(0) \|_{\energysp} \\
&\lesssim C(\eta)^{\frac{d-1}{2(d+1)}} N(0)^{\frac{d-1}{2(d+1)}}.
\end{align*}
Taking the $L^{2(d+1)/(d-1)}_t(J)$ norm of both sides of the previous estimate yields 
\begin{align}
 \| |D_x|^{s_p - 1/2} P_{\leq C(\eta)N(0)} S(t) \vec u(0) \|_{W(J)} \lesssim 
C(\eta)^{\frac{d-1}{2(d+1)}} \delta^{\frac{d-1}{2(d+1)}}. 
\label{ppn13}
\end{align}

Combining \eqref{ppn11}, \eqref{ppn12}, and \eqref{ppn13} we obtain 
\begin{align}
\| S(t) \vec u(0) \|_{X(J)} \lesssim \eta + C(\eta)^{\frac{2}{(p-1)(d+1)}} \delta^{\frac{2}{(p-1)(d+1)}}
+ C(\eta)^{\frac{d-1}{2(d+1)}} \delta^{\frac{d-1}{2(d+1)}}.
\end{align}

To estimate the Duhamel integral $\int_0^t \frac{\sin((t - \tau)\sqrt{-\Delta})}{\sqrt{-\Delta}} |u(\tau)|^{p-1} u(\tau) d\tau$,
we use Strichartz estimates and the chain rule for fractional derivatives
\begin{align*}
\left \| 
\int_0^t \frac{\sin((t - \tau)\sqrt{-\Delta})}{\sqrt{-\Delta}} |u(\tau)|^{p-1} u(\tau) d\tau
\right \|_{X(J)}
& \lesssim 
\| |D_x|^{s_p - 1/2} |u|^{p-1} u \|_{N(J)} \\
&\lesssim \| u \|_{S(J)}^{p-1} \| |D_x|^{s_p - 1/2} u \|_{W(J)} \\
&\lesssim \| u \|_{X(J)}^p.
\end{align*}
Combining this with the estimate for $\| S(t) \vec u(0) \|_{X(J)}$, we obtain
\begin{align*}
\| u \|_{X(J)} \lesssim \eta + C(\eta)^{\frac{2}{(p-1)(d+1)}} \delta^{\frac{2}{(p-1)(d+1)}}
+ C(\eta)^{\frac{d-1}{2(d+1)}} \delta^{\frac{d-1}{2(d+1)}}
 + \| u \|_{X(J)}^p.
\end{align*}
The proof is concluded by a standard continuity argument after taking $\delta$ sufficiently small.
\end{proof}

We now prove Proposition \ref{prop1}.  We first show that we gain 
an initial amount of regularity $\epsilon_p > 0$ where $\epsilon_p$ depends only on 
$s_p$.

\begin{lem}\label{lem6}
Let $u$ be as in Proposition \ref{prop5}.  Define 
\begin{align}\label{ep def}
\epsilon_p := 
\begin{cases}
\frac{1}{4} \min \left \{ 1-s_p, s_p - \frac{1}{2} \right \} &\mbox{ if } s_p > \frac{1}{2}, \\
\frac{1}{16} &\mbox{ if } s_p = \frac{1}{2}.
\end{cases}
\end{align}
Then for all $t \in (T_-,\infty)$, $\vec u(t) \in \dot H^{s_p + \epsilon_p}
\times \dot H^{s_p + \epsilon_p - 1}$ and 
\begin{align*}
\| \vec u(t)\|_{\dot H^{s_p + \epsilon_p}
	\times \dot H^{s_p + \epsilon_p - 1}} \lesssim N(t)^{\epsilon_p}.
\end{align*}
\end{lem}

\begin{proof}[Proof of Proposition \ref{lem6}]
Without loss of generality, we assume that $t_0 = 0$.  Let $\epsilon_p$
be as above.  We seek a frequency envelope $\{\alpha_k (0)\}_k$ so that 
\begin{align}
\| (P_k u(0), P_k \partial_t u(0)) \|_{\dot H^{s_p} \times \dot H^{s_p-1}} &\lesssim  \alpha_k(0)
\end{align}
with $\| \{ 2^{\epsilon_p k} \alpha_k(0) \} \|_{\ell^2} \lesssim N(0)^{\epsilon_p}$.

\begin{clm}\label{clm1}
There exists $\eta_0 > 0$ with the following property.  Let $0 < \eta < \eta_0$, and let $J = [-\delta/N(0), \delta/N(0)]$ where $\delta = \delta(\eta)$ is
chosen as in Lemma \ref{lem2}.  Let $\beta = s_p - 1/2$ if $s_p > 1/2$, and let $\beta = 1/4$ if $s_p = 1/2$.  Define a sequence of real numbers $\{a_k\}_k$ according to the following:
if $s_p > \frac{1}{2}$, define   
\begin{align}
a_k &:= 2^{s_pk} \| P_k u \|_{L^{\infty}_t(J ; L_x^2)} +  
2^{(s_p-1)k} \| P_k \partial_t u \|_{L^{\infty}_t(J ; L_x^2)} + 2^{\beta k} \| P_k u \|_{W(J)}
\end{align}
and if $s_p = \frac{1}{2}$, define
\begin{align}
a_k &:= 
\begin{cases}
2^{k/2} \| P_k u \|_{L^{\infty}_t(J ; L_x^2)} +  
2^{-k/2} \| P_k \partial_t u \|_{L^{\infty}_t(J ; L_x^2)} + 2^{k/4} \| P_k u \|_{L^{20/3}_tL^{5/2}_x(J \times \R^4)} &\mbox{ if } d = 4,\\
2^{k/2} \| P_k u \|_{L^{\infty}_t(J ; L_x^2)} +  
2^{-k/2} \| P_k \partial_t u \|_{L^{\infty}_t(J ; L_x^2)} + 2^{k/4} \| P_k u \|_{L^{6}_tL^{12/5}_x(J \times \R^5)} &\mbox{ if } d = 5.
\end{cases}
\end{align}
Define 
\begin{align*}
 a_k(0) &:= 2^{s_pk} \| P_k u(0) \|_{L^2_x} +  
2^{(s_p-1)k} \| P_k \partial_t u(0) \|_{L_x^2},
\end{align*}
and frequency envelopes $\{\alpha_k\}$ and $\{\alpha_k(0)\}$ by 
\begin{align}
\alpha_k &:= \sum_j 2^{-\frac{\beta}{2}|j-k|} a_j, \\
\alpha_k(0) &:= \sum_j 2^{-\frac{\beta}{2} |j-k|} a_j(0).
\end{align}
Then for $\eta_0$ sufficiently small we have
\begin{align}
a_k &\lesssim a_k(0) + \eta^2 \sum_j 2^{-\beta |j - k|} a_j, \\
\alpha_k &\lesssim \alpha_k(0). 
\end{align}
\end{clm}

\begin{proof}
We first handle the case $s_p > 1/2$.  The Strichartz estimates localized to frequency $\simeq 2^k$ read
\begin{align*}
a_k = 2^{s_pk} & \| P_k u \|_{L^{\infty}_t(J ; L_x^2)} +  
2^{(s_p-1)k} \| P_k \partial_t u \|_{L^{\infty}_t(J ; L_x^2)} + 2^{\beta k} \| P_k u \|_{W(J)} \\
&\lesssim a_k(0) + 2^{\beta k} \| P_k (|u|^{p-1}u) \|_{N(I)}. 
\end{align*}
Let $F(u) = |u|^{p-1}u$.  Choose $\rho \in (0,1)$ so that $\rho > 2 \beta$.  We observe that  
\begin{align*}
\| P_k F(P_{\leq k - 1} u) \|_{L_x^{\frac{2(d+1)}{d+3}}}
\lesssim& 2^{-\rho k} \| P_k |D_x|^{\rho} F(P_{\leq k -1} u ) \|_{L_x^{\frac{2(d+1)}{d+3}}} \\
\lesssim& 2^{-\rho k} \| P_{\leq k -1} u \|_{L_x^{\frac{(p-1)(d+1)}{2}}}^{p-1} \| |D_x|^{\rho} P_{\leq k -1} u \|_{L_x^{\frac{2(d+1)}{d-1}}} \\
\lesssim& 2^{-\rho k} \| u \|_{L_x^{\frac{(p-1)(d+1)}{2}}}^{p-1} \sum_{j \leq k -1} 
\| |D_x|^{\rho} P_j u \|_{L_x^{\frac{2(d+1)}{d-1}}} \\
\lesssim& \| u \|_{L_x^{\frac{(p-1)(d+1)}{2}}}^{p-1} \sum_{j \leq k -1} 2^{-(k - j)\rho}
\| P_j u \|_{L_x^{\frac{2(d+1)}{d-1}}}.
\end{align*}
By H\"older's inequality in time, we have 
\begin{align*}
2^{\beta k} \| P_k F(P_{\leq k - 1} u) \|_{N(J)} &\lesssim 
\| u \|_{S(J)}^{p-1} \sum_{j \leq k -1} 2^{-(\rho - \beta)(k - j)} 2^{\beta j} \| P_j u \|_{W(J)} \\
&\lesssim \eta^{p-1} \sum_{j \leq k -1} 2^{-\beta(k - j)} a_j.
\end{align*}
By the fundamental theorem of calculus 
\begin{align}
P_k ( F(u) - F(P_{\leq k - 1} u) ) 
= P_k \left ( \int_0^1 F' \left ( t P_{\geq k} u  + P_{\leq k -1} u \right ) dt P_{\geq k} u \right ).
\end{align}
Again by Holder's inequality in space and time and Bernstein's inequalities we obtain 
\begin{align*}
2^{\beta k} \| P_k ( F(u) - F(P_{\leq k - 1} u) ) \|_{N(I)} &\lesssim \| u \|_{S(J)}^{p-1} 2^{\beta k}
\| P_{\geq k} u \|_{W(J)} \\
\lesssim \eta^{p-1} \sum_{j \geq k} 2^{-\beta(j - k)} a_j.
\end{align*}
Combining the previous two inequalities, we obtain
\begin{align}
2^{\beta k} \| P_k F(u) \|_{N(I)} \lesssim \eta^{p-1} \sum_j 2^{-\beta |j - k|} a_j.
\end{align}
Thus, 
\begin{align*}
a_j \lesssim a_j(0) + \eta^{p-1} \sum_j 2^{-\beta |j - k|} a_j.  
\end{align*}

For the case $s_p = 1/2$, we use the following frequency localized Strichartz estimates: 
in $4d$
\begin{align*}
a_k = 2^{k/2}& \| P_k u \|_{L^{\infty}_t(J ; L_x^2)} +  
2^{-k/2} \| P_k \partial_t u \|_{L^{\infty}_t(J ; L_x^2)} + 2^{k/4} \| P_k u \|_{L^{20/3}_tL^{5/2}_x(J \times \R^4)} \\
&\lesssim a_k(0) + 2^{k/4} \| P_k  (|u|^{p-1}u) \|_{L^{20/11}_tL^{5/4}_x(J \times \R^4)}, 
\end{align*}
and in $5d$
\begin{align*}
a_k = 2^{k/2}& \| P_k u \|_{L^{\infty}_t(J ; L_x^2)} +  
2^{-k/2} \| P_k \partial_t u \|_{L^{\infty}_t(J ; L_x^2)} + 2^{k/4} \| P_k u \|_{L^{6}_tL^{12/5}_x(J \times \R^5)} \\
&\lesssim a_k(0) + 2^{k/4} \| P_k  (|u|^{p-1}u) \|_{L^{2}_tL^{4/3}_x(J \times \R^5)}.  
\end{align*}
The argument then proceeds exactly as in the case $s_p > 1/2$, and we obtain
\begin{align*}
a_j \lesssim a_j(0) + \eta^{p-1} \sum_j 2^{-\beta |j - k|} a_j.  
\end{align*}

We now estimate $\alpha_k$.  We have
\begin{align*}
\alpha_k &= \sum_j 2^{-\frac{\beta}{2} |j - k|} a_j \\
&\lesssim \sum_j 2^{-\frac{\beta}{2} |j - k|}  \left ( a_j(0) + \eta^{p-1} \sum_l 2^{-\beta |l - j|} a_l \right ) \\
&\lesssim \alpha_k(0) + \eta^{p-1} \sum_l \left ( \sum_j 2^{-\frac{\beta}{2} |j - k|} 2^{-\beta |l - j|} \right ) a_l. 
\end{align*}
By considering the cases $l \leq k$ and $l > k$, it is simple to verify
\begin{align*}
\sum_l \left ( \sum_j 2^{-\frac{\beta}{2} |j - k|} 2^{-\beta |l - j|} \right ) a_l \lesssim
\sum_l 2^{-\frac{\beta}{2}|l - k|} a_l,
\end{align*}
so that 
\begin{align*}
\alpha_k \lesssim \alpha_k(0) + \eta^{p-1} \alpha_k.
\end{align*}
Choosing $\eta_0$ sufficiently small implies 
\begin{align*}
\alpha_k \lesssim \alpha_k(0)
\end{align*}
as desired.
\end{proof}

Returning to the proof of Proposition \ref{prop5}, we see that the proof of the claim also yields the estimate
(with $\beta$ as in Claim \ref{clm1})
\begin{align}
2^{s_p k} \left \| 
P_k \int_0^{\delta /N(0)} \frac{e^{-it\sqrt{\Delta}}}{\sqrt{-\Delta}} |u(t)|^{p-1} u(t) dt
\right \|_{L^2_x} 
\lesssim \eta^{p-1} \sum_j 2^{-\beta |j - k|} a_j. \label{prop11}
\end{align}
The proof contains many elements from the proof of Proposition \ref{prop2}.   Let $\varphi \in C^\infty_0(\R^d)$ be radial such that $\varphi(x) = 1$ for $|x| \leq 1/8$ and $\varphi = 0$ for $|x| \geq 1/4$. 
We claim that given $\sigma \in (0, 1 - s_p]$, we have the estimate
\begin{align}
\int_{\delta / N(0)}^{T_2} 
\left \| 
\frac{e^{it \sqrt{-\Delta}}}{\sqrt{-\Delta}} \left ( 1 - \varphi \left (\frac{x}{t} \right ) \right ) |u(t)|^{p-1} u(t) dt
\right \|_{\dot H^{s_p + \sigma}} \lesssim N(0)^{\sigma} \delta^{-\sigma}, \label{prop12}
\end{align}
where the implied constant is uniform in $T_2 \in (0,T_+)$.  Indeed, by Sobolev embedding 
\begin{align*}
\left \| \frac{e^{it \sqrt{-\Delta}}}{\sqrt{-\Delta}}
\left ( 1 - \varphi \left (\frac{x}{t} \right ) \right ) |u(t)|^{p-1} u(t) 
\right \|_{\dot H^{s_p + \sigma}} &= \left \| 
\left ( 1 - \varphi \left (\frac{x}{t} \right ) \right ) |u(t)|^{p-1} u(t) 
\right \|_{\dot H^{s_p-1 + \sigma}} \\
&\lesssim \left \| \left ( 1 - \varphi \left (\frac{x}{t} \right ) \right ) |u(t)|^{p-1} u(t) 
\right \|_{L^q_x}  
\end{align*}
where $\frac{1}{q} = \frac{1}{2} + \frac{1-s_p - \sigma}{d}$. By radial Sobolev embedding
\begin{align*}
\left \| \left ( 1 - \varphi \left (\frac{x}{t} \right ) \right ) |u(t)|^{p-1}u(t)
\right \|_{L^q_x} &\lesssim |t|^{-1-\sigma} 
\left \| \left ( 1 - \varphi \left (\frac{x}{t} \right )  \right ) |x|^{1 + \sigma} |u(t)|^{p-1} u(t) 
\right \|_{L^q_x} \\
&\lesssim |t|^{-1 - \sigma} \| |x|^{\frac{1+\sigma}{p}} u(t) \|_{L_x^{qp}}^p \\
&\lesssim |t|^{-1 - \sigma} \| u \|_{\dot H^{s_p}}^p.
\end{align*}
Thus, 
\begin{align*}
\int_{\delta / N(0)}^{T_2} 
\left \| 
\frac{e^{it \sqrt{-\Delta}}}{\sqrt{-\Delta}} \left ( 1 - \varphi \left (\frac{x}{t} \right ) \right ) |u(t)|^{p-1} u(t) dt
\right \|_{\dot H^{s_p + \sigma}} &\lesssim 
\int_{\delta / N(0)}^{T_2} |t|^{-1 - \sigma} dt\\
&\lesssim N(0)^{\sigma} \delta^{-\sigma},
\end{align*}
uniformly in $T_2$.

Define $v$ as in the proof of Proposition \ref{prop2},i.e.
\begin{align*}
v(t) = u(t) + \frac{i}{\sqrt{-\Delta}} \partial_t u(t). 
\end{align*}
Then $\| v(0) \|_{\dot H^{s_p}} = \| \vec u(t) \|_{\energysp}$ and $v$ solves
\begin{align*}
\partial_t = -i \sqrt{-\Delta} v - \mu \frac{i}{\sqrt{-\Delta}} |u|^{p-1} u.
\end{align*}

We now estimate $\| P_k v(0) \|_{\dot H^{s_p}}$.  By the Duhamel formula,
for any $T_- < T_1 < 0 < T_2 < T_+$, 
\begin{align*}
\la P_k v(0) &, P_k v(0) \ra_{\dot H^{s_p}} \\
=& \left \la P_k \left ( e^{iT_2 \sqrt{-\Delta}} v(T_2) + i \mu \int_0^{T_2} \frac{e^{it \sqrt{-\Delta}}}{\sqrt{-\Delta}} |u(t)|^{p-1}u(t) dt
\right ), \right. \\
& \left. P_k \left (  e^{iT_1 \sqrt{-\Delta}} v(T_1) - i \mu \int^0_{T_1} \frac{e^{i\tau \sqrt{-\Delta}}}{
\sqrt{-\Delta}} |u(\tau)|^{p-1}u(\tau) d\tau
\right )  \right \ra_{\dot H^{s_p}}.
\end{align*}
As in the proof of Proposition \ref{prop2}, we can conclude via Lemma \ref{lem1} and \eqref{prop12} that 
\begin{align*}
\la& P_k v(0), P_k v(0) \ra_{\dot H^{s_p}} = \\
&\lim_{T_1 \rightarrow T_-} \lim_{T_2 \rightarrow T_+}
\left \la 
P_k  \left ( \int_0^{T_2} \frac{e^{it \sqrt{-\Delta}}}{\sqrt{-\Delta}} |u(t)|^{p-1}u(t) dt \right ), 
P_k  \left ( \int^0_{T_1} \frac{e^{i\tau \sqrt{-\Delta}}}{\sqrt{-\Delta}} |u(\tau)|^{p-1}u(\tau) d\tau \right )
\right \ra_{\dot H^{s_p}}.
\end{align*}
We split the Duhamel integrals into pieces and write the previous $\dot H^{s_p}$ pairing as a sum of $\dot H^{s_p}$ pairings
\begin{align*}
\la A + B, A' \ra + \la A , A' + B' \ra + \la B, B'\ra - \la A , A'\ra,
\end{align*}
where 
\begin{align}
A &= P_k \left ( \int_0^{\delta/N(0)} \frac{e^{it \sqrt{-\Delta}}}{\sqrt{-\Delta}} |u(t)|^{p-1}u(t) dt +
\int_{\delta/N(0)}^{T_2} \frac{e^{it \sqrt{-\Delta}}}{\sqrt{-\Delta}} \left ( 1 - \varphi \left ( \frac{x}{t} \right) \right )
 |u(t)|^{p-1}u(t) dt \right ), \\
B &= P_k \left (\int_{\delta/N(0)}^{T_2} \frac{e^{it \sqrt{-\Delta}}}{\sqrt{-\Delta}} \varphi \left ( \frac{x}{t} \right ) |u(t)|^{p-1}u(t) dt \right ),
\end{align}
and $A',B'$ are the analogous integrals in the negative time direction.  

We first estimate $A$.  By \eqref{prop11} and \eqref{prop12} we have the estimate (with $\beta$ defined in Claim \ref{clm1})
\begin{align*}
\| A \|_{\dot H^{s_p}} \lesssim 
\eta^{p-1} \sum_j 2^{-\beta |j - k|} a_j + 2^{-\epsilon_p k} b_k
\end{align*}
where $b_k = 2^{-k(\sigma(k) - \epsilon_p)} N(0)^{\sigma(k)}$, and we are free to choose 
$\sigma(k) \in (0,1-s_p]$ for each $k$. We choose  
\begin{align*}
\sigma(k) = 
\begin{cases}
1-s_p, &\mbox{if } 2^k \geq N(0), \\
\epsilon_p/2, &\mbox{if } 2^k < N(0).
\end{cases}
\end{align*}
Then 
\begin{align*}
b_k = 
\begin{cases}
2^{-k(1-s_p - \epsilon_p)} N(0)^{1-s_p}, &\mbox{if } 2^k \geq N(0), \\
2^{\epsilon_p k /2} N(0)^{\epsilon_p/2}, &\mbox{if } 2^k < N(0),
\end{cases}
\end{align*}
and (since $\epsilon_p \leq \frac{1}{4}(1 - s_p)$)
\begin{align}
\| \{ b_k \} \|_{\ell^2} \lesssim N(0)^{\epsilon_p}.
\end{align}
In summary, we have the estimate
\begin{align*}
\| A \|_{\dot H^{s_p}} \lesssim \eta^{p-1} \sum_j 2^{-\beta |j - k|} a_j + 2^{-\epsilon_p k} b_k, \quad 
\| \{b_k\}_k \|_{\ell^2} \lesssim N(0)^{\epsilon_p},
\end{align*}
uniformly in $T_2$.  The same estimate holds with $A$ replaced by $A'$ uniformly in $T_1$. This implies that 
\begin{align}
\la A , A' \ra \lesssim \eta^{2(p-1)} \left ( \sum_j 2^{-\beta |j - k|} a_j \right )^2+ 2^{-2\epsilon_p k} b_k^2 \label{prop14}
\end{align}
uniformly in $T_1$ and $T_2$.   

We now estimate $\la A, A' + B' \ra$.  By Lemma \ref{lem1}, 
\begin{align*}
- \mu i \int_{T_1}^0 \frac{e^{i \tau \sqrt{-\Delta}}}{\sqrt{-\Delta}} |u(\tau)|^{p-1} u(\tau) d\tau \rightharpoonup v(0)
\end{align*}
in $\dot H^{s_p}$ as $T_1 \rightarrow T_-$. As in the proof of Proposition \ref{prop2}, we have that 
\begin{align}
\left |\lim_{T_1 \rightarrow T_-} \lim_{T_2 \rightarrow T_+}
\la A , A' + B' \ra \right | 
&= \left |\lim_{T_1 \rightarrow T_-} \lim_{T_2 \rightarrow T_+} \left \la A , P_k v(0) \right \ra \right | \nonumber \\
&\lesssim a_k(0) \left ( 
\eta^{p-1} \sum_j 2^{-\beta |j - k|} a_j + 2^{-\epsilon_pk} b_k
\right ). \label{prop15}
\end{align}
By the same proof, the same estimate holds for $\la A + B , A' \ra$. 

Finally, we estimate the pairing $\la B, B' \ra$ which we write as 
\begin{align*}
\la &B , B' \ra \\ &= \int_{\delta/N(0)}^{T_2} \int^{-\delta/N(0)}_{T_1}
\left \la \varphi \left ( \frac{x}{|t|} \right ) |u(t)|^{p-1}u(t), P_k^2 \frac{e^{i(\tau - t)\sqrt{-\Delta}}}{(-\Delta)^{1-s_p}}
\varphi \left ( \frac{y}{|\tau|} \right ) |u(\tau)|^{p-1}u(\tau) \right \ra_{L^2} d\tau dt.  
\end{align*}
The operator $P_k^2 \frac{e^{i(\tau - t)\sqrt{-\Delta}}}{(-\Delta)^{1-s_p}}$ has a kernel given by 
\begin{align}
K_k ( x ) = K_k (|x|) = c \int_0^{\pi} \int_0^{+\infty} e^{i|x|\rho \cos \theta} e^{i(t- \tau)\rho} 
\psi^2 \left ( \frac{\rho}{2^k}\right ) \rho^{d-3+2s_p} d \rho \sin^{d-2} d\theta,
\end{align}
where  $\psi \in C^{\infty}_0$ is the Littlewood--Paley multiplier with $\supp \psi \subseteq
[1/2,2]$. As in the proof of Proposition \ref{prop2}, we have the estimate 
\begin{align}
|K_k(x - y)| \lesssim_L \frac{2^{(d-2+2s_p)k}}{\la 2^k |\tau - t|\ra^L},
\end{align}
for every $L \geq 0$ so that by Sobolev embedding  
\begin{align}
\left \la \varphi \left ( \frac{x}{|t|} \right ) |u(t)|^{p-1}u(t), P_k^2 \frac{e^{i(\tau - t)\sqrt{-\Delta}}}{(-\Delta)^{1-s_p}}
\varphi \left ( \frac{y}{|\tau|} \right ) |u(\tau)|^{p-1}u(\tau) \right \ra_{L^2}
\lesssim \frac{2^{(d-2+2s_p)k} |t \tau|^{d-\frac{2p}{p-1}}}{\la 2^k (|t| + |\tau|) \ra^L},
\end{align}
for all $\tau < -\delta/N(0) < \delta/N(0) < t$ and $L$ suffiently large.  

 If $2^k \geq N(0)$, we choose $L > d$ sufficiently large to make the following integral converge and
estimate 
\begin{align*}
\la B , B' \ra
&\lesssim_L \int_{\delta/N(0)}^{+\infty} \int^{-\delta/N(0)}_{-\infty} \frac{2^{(d-2+2s_p)k}
|t|^{d-\frac{2p}{p-1}} |\tau|^{d-\frac{2p}{p-1}}}{\la 2^k (|t| + |\tau|) \ra^L} d\tau dt \\
&\lesssim_L 2^{(d-2+2s_p-L)k} N(0)^{-2d + \frac{4p}{p-1} + L - 2} \\
&\lesssim_L 2^{(d-2+2s_p-L)k} N(0)^{-d + L} N(0)^{2p(s_0 - s_p)} \quad \left ( s_0 - s_p = \frac{2}{p-1} - \frac{d+2}{2p} = \frac{1}{p}(1-s_p) \right )\\
&\lesssim_L 2^{-2k(1-s_p)} N(0)^{2(1-s_p)}.
\end{align*}
If $2^k \leq N(0)$, we split the above integral into integration over $|t| + |\tau| \leq 2^{-k}$ and $|t| + |\tau| \geq 2^{-k}$.  In
the former region, we take $L = 0$ and conclude that 
\begin{align*}
\int \int_{|t| + |\tau| \leq 2^{-k}}\frac{2^{(d-2+2s_p)k}
|t|^{d-\frac{2p}{p-1}} |\tau|^{d-\frac{2p}{p-1}}}{\la 2^k (|t| + |\tau|) \ra^L} d\tau dt 
&\lesssim 2^{(d-2+2s_p)k} 2^{-\left (2d - \frac{4p}{p-1} + 2 \right )k} \\ 
&\lesssim 2^{2(s_p - 1)k}2^{2p \left ( \frac{4p}{p-1} - d -2 \right )k} \\
&\lesssim 2^{2(s_p - 1)k}2^{2p(s_0 - s_p)k} \\
&\lesssim 2^{2(s_p - 1)k}2^{2(1-s_p)k} \\
&\lesssim 1.
\end{align*}
In the region $|t| + |\tau| \geq 2^{-k}$ we choose $L > d$ sufficiently large to make the following integral converge and
estimate
\begin{align*}
\int \int_{|t| + |\tau| \geq 2^{-k}} \frac{2^{(d-2 + 2s_p)k}
|t|^{d-\frac{2p}{p-1}} |\tau|^{d-\frac{2p}{p-1}}}{\la 2^k (|t| + |\tau|) \ra^L} d\tau dt
&\lesssim_L 2^{(d-2+2s_p-L)k} 2^{\left( -2d + \frac{4p}{p-1} + L - 2\right) k} \lesssim 1. 
\end{align*}
Thus, we have the estimate 
\begin{align*}
 \la B , B' \ra \lesssim 1
\end{align*}
if $2^k \leq N(0)$.  In summary, we have proved that 
\begin{align}
\la B , B' \ra \lesssim \min \left ( 2^{-2k(1-s_p)}N(0)^{2(1-s_p)}, 1\right ) \label{prop13}.
\end{align}

Combining \eqref{prop14}, \eqref{prop15}, and \eqref{prop13}, we obtain the estimate 

\begin{align*}
a^2_k(0) \lesssim& a_k(0) \left ( \eta^{p-1} \sum_j 2^{-\beta |j - k|} a_j + 2^{-\epsilon_p k} b_k \right )
+ \eta^{2(p-1)} \left ( \sum_j 2^{-\beta |j - k|} a_j \right )^2 \\ &+ 2^{-2\epsilon_p k} b_k^2
+  \min \left ( 2^{-2(1-s_p)k}N(0)^{2(1-s_p)}, 1\right ).
\end{align*}
Hence, we have 
\begin{align*}
a_k(0) \lesssim& \eta^{p-1} \sum_j 2^{-\beta |j - k|} a_j + 2^{-\epsilon_p k} b_k +
\min \left ( 2^{-(1-s_p)k}N(0)^{1-s_p}, 1\right ).
\end{align*}
Using the definitions of $\alpha_k$ and $\alpha_k(0)$ and Claim \ref{clm1} we obtain
\begin{align*}
\alpha_k(0) &\lesssim \eta^{p-1} \alpha_k + \sum_j 2^{-\frac{\beta}{2}|j-k|} 2^{-\epsilon_p j} b_j 
+ \sum_j 2^{-\frac{\beta}{2}|j-k|} 2^{-\epsilon_p j} c_j \\ &\lesssim
\eta^{p-1} \alpha_k(0) + \sum_j 2^{-\frac{\beta}{2}|j-k|} 2^{-\epsilon_p j} b_j
+ \sum_j 2^{-\frac{\beta}{2}|j-k|} 2^{-\epsilon_p j} c_j
\end{align*}
where $c_j 
= N(0)^{\epsilon_p} \min  \left ( 2^{-(1 - s_p - \epsilon_p)j} N(0)^{1-s_p-\epsilon_p}, 2^{\epsilon_p j} N(0)^{-\epsilon_p} \right )$.  Note that
\begin{align*}
\| \{ c_j \} \|_{\ell^2} \lesssim N(0)^{\epsilon_p}
\end{align*}
By choosing $\eta$ sufficiently 
small, we have that 
\begin{align*}
\alpha_k(0) \lesssim \sum_j 2^{-\frac{\beta}{2}|j-k|} 2^{-\epsilon_p j} b_j +
 \sum_j 2^{-\frac{\beta}{2}|j-k|} 2^{-\epsilon_p j} c_j.
\end{align*}
By Minkowski's inequality and the fact that $\epsilon_p \leq \beta/4$, we conclude that 
\begin{align*}
\| 2^{k\epsilon_p} \alpha_k(0) \|_{\ell^2} \lesssim \| \{b_k\} \|_{\ell^2} + 
\| \{c_k\} \|_{\ell^2} \lesssim N(0)^{\epsilon_p}.
\end{align*}
This finishes the proof of Lemma \ref{lem6}.

\end{proof}

We will now show that we gain regularity in increments determined by $\epsilon_p$ and the distance we are from having $\dot H^1 \times L^2$
regularity.  More precisely we will prove the following. 

\begin{lem}\label{lem7}
Let $u$ be as in Proposition \ref{prop5}.  Assume additionally that for all $t \in (T_-,\infty)$, $\vec u(t) \in \dot H^s \times \dot H^{s-1}$.  Define 
\begin{align*}
\epsilon_s = \min\{1-s, \epsilon_p\}.
\end{align*}
Then for every $0 < \epsilon < \epsilon_s$ and 
for all $t \in (T_-,\infty)$, $\vec u(t) \in \dot H^{s + \epsilon} \times 
\dot H^{s + \epsilon}$ and 
\begin{align*}
\| \vec u(t) \|_{\dot H^{s + \epsilon} \times \dot H^{s + \epsilon - 1}}
\lesssim N(t)^{s + \epsilon - s_p}.
\end{align*}  
\end{lem} 
		
\begin{proof}[Proof of Lemma \ref{lem7}]
		
The proof of Lemma \ref{lem7} is very similar  to the proof of Lemma \ref{lem6} so we will omit some details.  Let $\epsilon$ be as above.   
	Without loss of generality, we assume that $t_0 = 0$.  We seek a frequency envelope $\{\alpha_k (0)\}_k$ so that 
	\begin{align}
	\| (P_k u(0), P_k \partial_t u(0)) \|_{\dot H^{s} \times \dot H^{s-1}} &\lesssim  \alpha_k(0)
	\end{align}
	with $\| \{ 2^{\epsilon k} \alpha_k(0) \} \|_{\ell^2} \lesssim N(0)^{s+\epsilon - s_p}$.

	\begin{clm}\label{clm1s}
		There exists $\eta_0 > 0$ with the following property.  Let $0 < \eta < \eta_0$, and let $J = [-\delta/N(0), \delta/N(0)]$ where $\delta = \delta(\eta)$ is
		chosen as in Lemma \ref{lem2}.  Let $\beta$ be as in Lemma \ref{clm1} and define 
		\begin{align*}
		\gamma = \beta + (s - s_p).
		\end{align*}
		  Define a sequence of real numbers $\{a_k\}_k$ according to the following:
		if $s_p > \frac{1}{2}$, define   
		\begin{align}
		a_k &:= 2^{s k} \| P_k u \|_{L^{\infty}_t(J ; L_x^2)} +  
		2^{(s-1)k} \| P_k \partial_t u \|_{L^{\infty}_t(J ; L_x^2)} + 2^{\gamma k} \| P_k u \|_{W(J)}
		\end{align}
		and if $s_p = \frac{1}{2}$, define
		\begin{align}
		a_k &:= 
		\begin{cases}
		2^{sk} \| P_k u \|_{L^{\infty}_t(J ; L_x^2)} +  
		2^{-sk} \| P_k \partial_t u \|_{L^{\infty}_t(J ; L_x^2)} + 2^{k(s-1/4)} \| P_k u \|_{L^{20/3}_tL^{5/2}_x(J \times \R^4)} &\mbox{ if } d = 4,\\
		2^{sk} \| P_k u \|_{L^{\infty}_t(J ; L_x^2)} +  
		2^{-sk} \| P_k \partial_t u \|_{L^{\infty}_t(J ; L_x^2)} + 2^{k(s-1/4)} \| P_k u \|_{L^{6}_tL^{12/5}_x(J \times \R^5)} &\mbox{ if } d = 5.
		\end{cases}
		\end{align}
		Define 
		\begin{align*}
		a_k(0) &:= 2^{s_pk} \| P_k u(0) \|_{L^2_x} +  
		2^{(s_p-1)k} \| P_k \partial_t u(0) \|_{L_x^2},
		\end{align*}
		and frequency envelopes $\{\alpha_k\}$ and $\{\alpha_k(0)\}$ by 
		\begin{align}
		\alpha_k &:= \sum_j 2^{-\frac{\beta}{2}|j-k|} a_j, \\
		\alpha_k(0) &:= \sum_j 2^{-\frac{\beta}{2} |j-k|} a_j(0).
		\end{align}
		Then for $\eta_0$ sufficiently small we have
		\begin{align}
		a_k &\lesssim a_k(0) + \eta^2 \sum_j 2^{-\beta |j - k|} a_j, \\
		\alpha_k &\lesssim \alpha_k(0). 
		\end{align}
	\end{clm}
	
	We remark that the additional regularity assumption in Lemma \ref{lem7}
	and Strichartz estimates imply that the numbers $\alpha_k$ and $\alpha_k(0)$ are finite for each $k$. 
	
	\begin{proof}
		We first handle the case $s_p > 1/2$.  The Strichartz estimates at the $\dot H^s \times \dot H^{s-1}$ regularity localized to frequency $\simeq 2^k$ read
		\begin{align*}
		a_k = 2^{s k} & \| P_k u \|_{L^{\infty}_t(J ; L_x^2)} +  
		2^{(s-1)k} \| P_k \partial_t u \|_{L^{\infty}_t(J ; L_x^2)} + 2^{\gamma k} \| P_k u \|_{W(J)} \\
		&\lesssim a_k(0) + 2^{\gamma k} \| P_k (|u|^{p-1}u) \|_{N(J)}. 
		\end{align*}
		As before, we let $F(u) = |u|^{p-1}u$.  We first observe by Bernstein's inequalities that  
		\begin{align*}
		\| P_k F(P_{\leq k - 1} u) \|_{L_x^{\frac{2(d+1)}{d+3}}}
		\lesssim& 2^{-k} \| P_k |D_x| F(P_{\leq k -1} u ) \|_{L_x^{\frac{2(d+1)}{d+3}}} \\
		\lesssim& 2^{-k} \| P_{\leq k -1} u \|_{L_x^{\frac{(p-1)(d+1)}{2}}}^{p-1} \| |D_x| P_{\leq k -1} u \|_{L_x^{\frac{2(d+1)}{d-1}}} \\
		\lesssim& 2^{-k} \| u \|_{L_x^{\frac{(p-1)(d+1)}{2}}}^{p-1} \sum_{j \leq k -1} 
		\| |D_x| P_j u \|_{L_x^{\frac{2(d+1)}{d-1}}} \\
		\lesssim& \| u \|_{L_x^{\frac{(p-1)(d+1)}{2}}}^{p-1} \sum_{j \leq k -1} 2^{-(k - j)}
		\| P_j u \|_{L_x^{\frac{2(d+1)}{d-1}}}.
		\end{align*}
 By H\"older's inequality in time, we have 
		\begin{align*}
		2^{\gamma k} \| P_k F(P_{\leq k - 1} u) \|_{N(J)} &\lesssim 
		\| u \|_{S(J)}^{p-1} \sum_{j \leq k -1} 2^{-(1 - \gamma)(k - j)} 2^{\gamma j} \| P_j u \|_{W(J)} \\
		&\lesssim \eta^{p-1} \sum_{j \leq k -1} 2^{-(1-\gamma)(k - j)} a_j.
		\end{align*}
		By the fundamental theorem of calculus 
		\begin{align}
		P_k ( F(u) - F(P_{\leq k - 1} u) ) 
		= P_k \left ( \int_0^1 F' \left ( t P_{\geq k} u  + P_{\leq k -1} u \right ) dt P_{\geq k} u \right ).
		\end{align}
		Again by Holder's inequality in space and time and Bernstein's inequalities we obtain 
		\begin{align*}
		2^{\gamma k} \| P_k ( F(u) - F(P_{\leq k - 1} u) ) \|_{N(I)} &\lesssim \| u \|_{S(J)}^{p-1} 2^{\gamma k}
		\| P_{\geq k} u \|_{W(J)} \\
		\lesssim \eta^{p-1} \sum_{j \geq k} 2^{-\gamma(j - k)} a_j.
		\end{align*}
		Combining the previous two inequalities and using the fact that 
		$\gamma \geq \beta$ and $(1-\gamma) \geq \beta$ since $s_p$ and 
		$s$ are less than 1, we see that
		\begin{align}
		2^{\gamma k} \| P_k F(u) \|_{N(I)} \lesssim \eta^{p-1} \sum_j 2^{-\beta |j - k|} a_j.
		\end{align}
		Thus, 
		\begin{align*}
		a_j \lesssim a_j(0) + \eta^{p-1} \sum_j 2^{-\beta |j - k|} a_j.  
		\end{align*}
		As in the proof of Claim \ref{clm1}, we then deduce that 
		\begin{align*}
		\alpha_k \lesssim \alpha_k(0)
		\end{align*}
as desired.

\end{proof}
	
	Returning to the proof of Lemma \ref{lem7}, we see that the proof of the claim also yields the estimate
	(with $\beta$ as in Claim \ref{clm1})
	\begin{align}
	2^{s k} \left \| 
	P_k \int_0^{\delta /N(0)} \frac{e^{-it\sqrt{\Delta}}}{\sqrt{-\Delta}} |u(t)|^{p-1} u(t) dt
	\right \|_{L^2_x} 
	\lesssim \eta^{p-1} \sum_j 2^{-\beta |j - k|} a_j. \label{prop11s}
	\end{align}
Let $\varphi \in C^\infty_0(\R^d)$ be radial such that $\varphi(x) = 1$ for $|x| \leq 1/8$ and $\varphi = 0$ for $|x| \geq 1/4$. 
	By \eqref{prop12} we have for all $\sigma \in (0,1-s]$
	\begin{align}
	\int_{\delta / N(0)}^{T_2} 
	\left \| 
	\frac{e^{it \sqrt{-\Delta}}}{\sqrt{-\Delta}} \left ( 1 - \varphi \left (\frac{x}{t} \right ) \right ) |u(t)|^{p-1} u(t) dt
	\right \|_{\dot H^{s + \sigma}} \lesssim N(0)^{s - s_p + \sigma} \delta^{-(s - s_p + \sigma)}, \label{prop12s}
	\end{align}
	where the implied constant is uniform in $T_2 \in (0,T_+)$.  
	We define $v$ via 
	\begin{align*}
	v(t) = u(t) + \frac{i}{\sqrt{-\Delta}} \partial_t u(t). 
	\end{align*}
	Then $\| v(0) \|_{\dot H^{s}} = \| \vec u(t) \|_{\dot H^s \times \dot H^{s-1}}$ and $v$ solves
	\begin{align*}
	\partial_t = -i \sqrt{-\Delta} v - \mu \frac{i}{\sqrt{-\Delta}} |u|^{p-1} u.
	\end{align*}
	
	We now estimate $\| P_k v(0) \|_{\dot H^{s}}$.  By the Duhamel formula,
	for any $T_- < T_1 < 0 < T_2 < T_+$, 
	\begin{align*}
	\la P_k v(0) &, P_k v(0) \ra_{\dot H^{s}} \\
	=& \left \la P_k \left ( e^{iT_2 \sqrt{-\Delta}} v(T_2) + i \mu \int_0^{T_2} \frac{e^{it \sqrt{-\Delta}}}{\sqrt{-\Delta}} |u(t)|^{p-1}u(t) dt
	\right ), \right. \\
	& \left. P_k \left (  e^{iT_1 \sqrt{-\Delta}} v(T_1) - i \mu \int^0_{T_1} \frac{e^{i\tau \sqrt{-\Delta}}}{
		\sqrt{-\Delta}} |u(\tau)|^{p-1}u(\tau) d\tau
	\right )  \right \ra_{\dot H^{s}}.
	\end{align*}
	As in the proof of Proposition \ref{prop2}, we can conclude via Lemma \ref{lem1}, \eqref{prop12s} and the fact that we are localizing to frequency $2^k$ that 
	\begin{align*}
	\la& P_k v(0), P_k v(0) \ra_{\dot H^{s}} = \\
	&\lim_{T_1 \rightarrow T_-} \lim_{T_2 \rightarrow T_+}
	\left \la 
	P_k  \left ( \int_0^{T_2} \frac{e^{it \sqrt{-\Delta}}}{\sqrt{-\Delta}} |u(t)|^{p-1}u(t) dt \right ), 
	P_k  \left ( \int^0_{T_1} \frac{e^{i\tau \sqrt{-\Delta}}}{\sqrt{-\Delta}} |u(\tau)|^{p-1}u(\tau) d\tau \right )
	\right \ra_{\dot H^{s}}.
	\end{align*}
	We split the Duhamel integrals into pieces and write the previous $\dot H^{s}$ pairing as a sum of $\dot H^{s}$ pairings
	\begin{align*}
	\la A + B, A' \ra + \la A , A' + B' \ra + \la B, B'\ra - \la A , A'\ra,
	\end{align*}
	where 
	\begin{align}
	A &= P_k \left ( \int_0^{\delta/N(0)} \frac{e^{it \sqrt{-\Delta}}}{\sqrt{-\Delta}} |u(t)|^{p-1}u(t) dt +
	\int_{\delta/N(0)}^{T_2} \frac{e^{it \sqrt{-\Delta}}}{\sqrt{-\Delta}} \left ( 1 - \varphi \left ( \frac{x}{t} \right) \right )
	|u(t)|^{p-1}u(t) dt \right ), \\
	B &= P_k \left (\int_{\delta/N(0)}^{T_2} \frac{e^{it \sqrt{-\Delta}}}{\sqrt{-\Delta}} \varphi \left ( \frac{x}{t} \right ) |u(t)|^{p-1}u(t) dt \right ),
	\end{align}
	and $A',B'$ are the analogous integrals in the negative time direction.  
	
	We first estimate $A$.  By \eqref{prop11s} and \eqref{prop12s} we have the estimate (with $\beta$ defined in Claim \ref{clm1})
	\begin{align*}
	\| A \|_{\dot H^{s}} \lesssim 
	\eta^{p-1} \sum_j 2^{-\beta |j - k|} a_j + 2^{-\epsilon k} b_k
	\end{align*}
	where $b_k = 2^{-k(\sigma(k) - \epsilon)} N(0)^{s - s_p +\sigma(k)}$, and we are free to choose 
	$\sigma(k) \in (0,1-s]$ for each $k$. We choose  
	\begin{align*}
	\sigma(k) = 
	\begin{cases}
	1-s, &\mbox{if } 2^k \geq N(0), \\
	\epsilon/2, &\mbox{if } 2^k < N(0).
	\end{cases}
	\end{align*}
	Then 
	\begin{align*}
	b_k = 
	\begin{cases}
	2^{-k(1-s - \epsilon)} N(0)^{1-s_p} &\mbox{if } 2^k \geq N(0), \\
	2^{\epsilon k/2} N(0)^{s - s_p + \epsilon/2} &\mbox{if } 2^k < N(0).
	\end{cases}
	\end{align*}
	Since $\epsilon < 1 -s$ we have $(b_k) \in \ell^2$ and 
	\begin{align}
	\| \{ b_k \} \|_{\ell^2} \lesssim N(0)^{s - s_p + \epsilon}.
	\end{align}
	In summary, we have the estimate
	\begin{align*}
	\| A \|_{\dot H^{s}} \lesssim \eta^{p-1} \sum_j 2^{-\beta |j - k|} a_j + 2^{-\epsilon k} b_k, \quad 
	\| \{b_k\}_k \|_{\ell^2} \lesssim N(0)^{s - s_p + \epsilon},
	\end{align*}
	uniformly in $T_2$.  The same estimate holds with $A$ replaced by $A'$ uniformly in $T_1$. This implies that 
	\begin{align}
	\la A , A' \ra \lesssim \eta^{2(p-1)} \left ( \sum_j 2^{-\beta |j - k|} a_j \right )^2+ 2^{-2\epsilon k} b_k^2 \label{prop14s}
	\end{align}
	uniformly in $T_1$ and $T_2$.   
	
	We now estimate $\la A, A' + B' \ra$.  By Lemma \ref{lem1}, 
	\begin{align*}
	- \mu i \int_{T_1}^0 \frac{e^{i \tau \sqrt{-\Delta}}}{\sqrt{-\Delta}} |u(\tau)|^{p-1} u(\tau) d\tau \rightharpoonup v(0)
	\end{align*}
	in $\dot H^{s_p}$ as $T_1 \rightarrow T_-$. As in the proof of Lemma \ref{lem6}, we have that 
	\begin{align}
	\left |\lim_{T_1 \rightarrow T_-} \lim_{T_2 \rightarrow T_+}
	\la A , A' + B' \ra \right | 
	&= \left |\lim_{T_1 \rightarrow T_-} \lim_{T_2 \rightarrow T_+} \left \la A , P_k v(0) \right \ra \right | \nonumber \\
	&\lesssim a_k(0) \left ( 
	\eta^{p-1} \sum_j 2^{-\beta |j - k|} a_j + 2^{-\epsilon_pk} b_k
	\right ). \label{prop15s}
	\end{align}
	By the same proof, the same estimate holds for $\la A + B , A' \ra$. 
	
	Finally, we estimate the pairing $\la B, B' \ra$ which we write as 
	\begin{align*}
	\la &B , B' \ra \\ &= \int_{\delta/N(0)}^{T_2} \int^{-\delta/N(0)}_{T_1}
	\left \la \varphi \left ( \frac{x}{|t|} \right ) |u(t)|^{p-1}u(t), P_k^2 \frac{e^{i(\tau - t)\sqrt{-\Delta}}}{(-\Delta)^{1-s_p}}
	\varphi \left ( \frac{y}{|\tau|} \right ) |u(\tau)|^{p-1}u(\tau) \right \ra_{L^2} d\tau dt.  
	\end{align*}
	
	The operator $P_k^2 \frac{e^{i(\tau - t)\sqrt{-\Delta}}}{(-\Delta)^{1-s}}$ has a kernel given by 
	\begin{align}
	K_k ( x ) = K_k (|x|) = c \int_0^{\pi} \int_0^{+\infty} e^{i|x|\rho \cos \theta} e^{i(t- \tau)\rho} 
	\psi^2 \left ( \frac{\rho}{2^k}\right ) \rho^{d-3+2s} d \rho \sin^{d-2} d\theta,
	\end{align}
	where  $\psi \in C^{\infty}_0$ is the Littlewood--Paley multiplier with $\supp \psi \subseteq
	[1/2,2]$. As in the proof of Proposition \ref{prop2}, we have the estimate 
	\begin{align}
	|K_k(x - y)| \lesssim_L \frac{2^{(d-2+2s)k}}{\la 2^k |\tau - t|\ra^L},
	\end{align}
	for every $L \geq 0$ so that by Sobolev embedding  
	\begin{align}
	\left \la \varphi \left ( \frac{x}{|t|} \right ) |u(t)|^{p-1}u(t), P_k^2 \frac{e^{i(\tau - t)\sqrt{-\Delta}}}{(-\Delta)^{1-s}}
	\varphi \left ( \frac{y}{|\tau|} \right ) |u(\tau)|^{p-1}u(\tau) \right \ra_{L^2}
	\lesssim \frac{2^{(d-2+2s)k} |t \tau|^{d-\frac{2p}{p-1}}}{\la 2^k (|t| + |\tau|) \ra^L},
	\end{align}
	for all $\tau < -\delta/N(0) < \delta/N(0) < t$ and $L$ sufficiently large.  
	
	If $2^k \geq N(0)$, we choose $L > d$ sufficiently large to make the following integral converge and
	estimate 
	\begin{align*}
	\la B , B' \ra
	&\lesssim_L \int_{\delta/N(0)}^{+\infty} \int^{-\delta/N(0)}_{-\infty} \frac{2^{(d-2+2s)k}
		|t|^{d-\frac{2p}{p-1}} |\tau|^{d-\frac{2p}{p-1}}}{\la 2^k (|t| + |\tau|) \ra^L} d\tau dt \\
	&\lesssim_L 2^{(d-2+2s-L)k} N(0)^{-2d + \frac{4p}{p-1} + L - 2} \\
	&\lesssim_L 2^{(d-2+2s-L)k} N(0)^{-d + L} N(0)^{2p(s_0 - s_p)} \quad \left ( s_0 - s_p = \frac{2}{p-1} - \frac{d+2}{2p} = \frac{1}{p}(1-s_p) \right )\\
	&\lesssim_L 2^{-2k(1-s)} N(0)^{2(1-s_p)}.
	\end{align*}
	If $2^k \leq N(0)$, we split the above integral into integration over $|t| + |\tau| \leq 2^{-k}$ and $|t| + |\tau| \geq 2^{-k}$.  In
	the former region, we take $L = 0$ and conclude that 
	\begin{align*}
	\int \int_{|t| + |\tau| \leq 2^{-k}}\frac{2^{(d-2+2s)k}
		|t|^{d-\frac{2p}{p-1}} |\tau|^{d-\frac{2p}{p-1}}}{\la 2^k (|t| + |\tau|) \ra^L} d\tau dt 
	&\lesssim 2^{(d-2+2s)k} 2^{-\left (2d - \frac{4p}{p-1} + 2 \right )k} \\ 
	&\lesssim 2^{2(s - 1)k}2^{\left ( \frac{4p}{p-1} - d -2 \right )k} \\
	&\lesssim 2^{2(s - 1)k}2^{2p(s_0 - s_p)k} \\
	&\lesssim 2^{2(s - 1)k}2^{2(1-s_p)k} \\
	&\lesssim 2^{2(s - s_p)} \\
	&\lesssim N(0)^{2(s-s_p)}.
	\end{align*}
	In the region $|t| + |\tau| \geq 2^{-k}$ we choose $L > d$ sufficiently large to make the following integral converge and
	estimate
	\begin{align*}
	\int \int_{|t| + |\tau| \geq 2^{-k}} \frac{2^{(d-2 + 2s_p)k}
		|t|^{d-\frac{2p}{p-1}} |\tau|^{d-\frac{2p}{p-1}}}{\la 2^k (|t| + |\tau|) \ra^L} d\tau dt
	&\lesssim_L 2^{(d-2+2s_p-L)k} 2^{\left( -2d + \frac{4p}{p-1} + L - 2\right) k} \lesssim 2^{2(s-s_p)} \lesssim N(0)^{2(s-s_p)}. 
	\end{align*}
	Thus, we have the estimate 
	\begin{align*}
	\la B , B' \ra \lesssim N(0)^{s-s_p}
	\end{align*}
	if $2^k \leq N(0)$.  In summary, we have proved that 
	\begin{align}
	\la B , B' \ra \lesssim 2^{-2\epsilon} c_k^2 \label{prop13s}
	\end{align}
	where 
	\begin{align*}
	c_k := 
	\begin{cases}
	2^{-(1-s-\epsilon)k} N(0)^{1-s_p} &\mbox{ if } 2^k > N(0), \\
	2^{\epsilon k} N(0)^{s-s_p} &\mbox{ if } 2^k \leq N(0).
	\end{cases}
	\end{align*}
	In particular, since $\epsilon < 1 - s$ we see that $\{ c_j \} \in \ell^2$ and
	\begin{align}\label{cj ell2}
	\| \{ c_j \} \|_{\ell^2} \lesssim N(0)^{s-s_p + \epsilon}.
	\end{align}
	
	Combining \eqref{prop14s}, \eqref{prop15s}, and \eqref{prop13s}, we obtain the estimate 
	
	\begin{align*}
	a^2_k(0) \lesssim& a_k(0) \left ( \eta^{p-1} \sum_j 2^{-\beta |j - k|} a_j + 2^{-\epsilon k} b_k \right )
	+ \eta^{2(p-1)} \left ( \sum_j 2^{-\beta |j - k|} a_j \right )^2 \\ &+ 2^{-2\epsilon k} b_k^2
	+  2^{-2\epsilon k} c_k^2
	\end{align*}
	Hence, we have 
	\begin{align*}
	a_k(0) \lesssim& \eta^{p-1} \sum_j 2^{-\beta |j - k|} a_j + 2^{-\epsilon k} b_k +
	2^{-\epsilon k} c_k.
	\end{align*}
	Using the definitions of $\alpha_k$ and $\alpha_k(0)$ and Claim \ref{clm1} we obtain
	\begin{align*}
	\alpha_k(0) &\lesssim \eta^{p-1} \alpha_k + \sum_j 2^{-\frac{\beta}{2}|j-k|} 2^{-\epsilon_p j} b_j +
	 \sum_j 2^{-\frac{\beta}{2}|j-k|} 2^{-\epsilon_p j} c_j \\ &\lesssim
	\eta^{p-1} \alpha_k(0) + \sum_j 2^{-\frac{\beta}{2}|j-k|} 2^{-\epsilon_p j} b_j +
	 \sum_j 2^{-\frac{\beta}{2}|j-k|} 2^{-\epsilon_p j} c_j
	\end{align*}
	By choosing $\eta$ sufficiently 
	small, we have that 
	\begin{align*}
	\alpha_k(0) \lesssim \sum_j 2^{-\frac{\beta}{2}|j-k|} 2^{-\epsilon j} b_j
	+ \sum_j 2^{-\frac{\beta}{2}|j-k|} 2^{-\epsilon j} c_j.
	\end{align*}
	By Minkowski's inequality and the fact that $\epsilon \leq \beta/4$, we conclude that 
	\begin{align*}
	\| 2^{k\epsilon_p} \alpha_k(0) \|_{\ell^2} \lesssim \| \{b_k\} \|_{\ell^2} + 
	\| \{c_k\} \|_{\ell^2} \lesssim N(0)^{s-s_p + \epsilon}.
	\end{align*}
	This finishes the proof of Lemma \ref{lem7}.
\end{proof}

From Lemma \ref{lem6} and Lemma \ref{lem7}, we immediately deduce 
Proposition \ref{prop5} and Proposition \ref{prop1}.

\qed

\subsection{Rigidity argument for the frequency cascade case}

Based on Proposition \ref{prop2}, we can quickly show that in the frequency cascade case 
\begin{align*}
\liminf_{t \rightarrow +\infty} N(t) = 0, 
\end{align*}
a solution to \eqref{nlw} with the compactness property must be identically 0. 

\begin{ppn}\label{sol}\label{noenergycas}
Let $u \in C(\R ; (\dot H^{s_p} \times \dot H^{s_p-1}) \cap (\dot H^1 \times L^2))$ be a solution to \eqref{nlw} on 
$I_{\max}(u) = (T_-, +\infty)$ with the compactness property such that 
\begin{align}
\liminf_{t \rightarrow +\infty} N(t) = 0.
\end{align}
Then $u = 0$.
\end{ppn}

\begin{proof}
By Proposition \ref{prop2}, we have the bound
\begin{align}
\| \vec u(t) \|_{\dot H^1 \times L^2} \lesssim N(t)^{1-s_p}, \quad t \in I_{\max}(u).  \label{ppn31}
\end{align}
By Sobolev embedding and interpolation, we have for some $\theta \in (0,1)$ 
\begin{align}
\| u(t) \|_{L^{p+1}_x} \lesssim \| \vec u(t) \|_{\energysp}^{1-\theta} \| \vec u(t) \|_{\dot H^1 \times L^2}^{\theta}
\lesssim N(t)^{\theta(1-s_p)}, \quad t \in I_{\max}(u). \label{ppn32}
\end{align}
By \eqref{ppn31} and \eqref{ppn32} the energy $E(\vec u)$ is well defined, conserved, and by our assumption on $N(t)$ it follows
that  
\begin{align}
\vec E(u(0)) = \liminf_{t \rightarrow +\infty} E(\vec u(t)) \lesssim 
\liminf_{t \rightarrow +\infty} \left (\| \vec u(t) \|_{\dot H^1 \times L^2}^2 + \| u(t) \|_{L^{p+1}_x}^{p+1} \right )
= 0
\end{align}
Hence $E(\vec u) \leq 0$.  In the defocusing case $\mu = 1$, this immediately implies that 
$u \equiv 0$.  In the focusing case, this implies by Proposition \ref{prop3} that either $I_{\max}$ is finite or 
$u \equiv 0$.  Since $I_{\max} = (T_-,+\infty)$, we have that $u \equiv 0$ in the focusing case as well. 
\end{proof}

\section{No Soliton--like Solutions Via a Virial Identity}

We now consider a solution $u$ of \eqref{nlw} with the compactness property that is soliton--like, i.e. $N(t) \equiv 1$.  
We first show that in this case, $\vec u(t) \in C(\R; \dot H^{1+s_p} \times \dot H^{s_p})$ with a bound
uniform in $t$.  This implies by interpolation that the trajectory $\{ \vec u(t) : t \in \R\}$ is precompact in $\dot H^1 \times
L^2$.  In the second part of this section, we conclude that $u \equiv 0$.

\subsection{Higher regularity for soliton--like solutions}

In this section, we prove the following.

\begin{ppn}\label{prop4}
Let $u$ be a solution to \eqref{nlw} on $\R$ with the compactness property such that $N(t) \equiv 1$.  Then there exists $\epsilon > 0$ such that for every $t \in \R$,
$\vec u(t) \in \dot H^{1+\epsilon} \times \dot H^{\epsilon}$ and there exists a constant $C > 0$ independent of $t$ such that
\begin{align}
\| \vec u(t) \|_{\dot H^{1+\epsilon} \times \dot H^{\epsilon}} \leq C.
\end{align}
\end{ppn}

\begin{proof}[Proof of Proposition \ref{prop4}]
The proof is essentially the same as the proof of Proposition \ref{prop1}.  It is simpler now since
$N(t) \equiv 1$ and we have the a priori bound (from 
Proposition \ref{prop2})
\begin{align*}
\sup_{t \in \R} \| \vec u(t) \|_{\dot H^s \times \dot H^{s-1}} \lesssim 1, \quad s \in [s_p, 1].
\end{align*}
Let $a_k, a_k(0), \alpha_k$ and $\alpha_k(0)$ be as in Claim \ref{clm1s}
with $s = 1$, and let $\delta$ be as in Claim \ref{clm1s}.  We recall the estimate
\begin{align}
2^{k} \left \| 
P_k \int_0^{\delta} \frac{e^{-it\sqrt{\Delta}}}{\sqrt{-\Delta}} |u(t)|^{p-1} u(t) dt
\right \|_{L^2_x} 
\lesssim \eta^{p-1} \sum_j 2^{-\beta |j - k|} a_j. \label{prop43}
\end{align}
Let $\varphi \in C^\infty_0(\R^d)$ such that $\varphi(x) = 1$ for $|x| \leq 1/8$ and $\varphi = 0$ for $|x| \geq 1/4$. 
We claim that given $\sigma \in (0, 2 - s_p]$, we have the estimate
\begin{align}
\int_{\delta}^{T_2} 
\left \| \frac{e^{it\sqrt{-\Delta}}}{\sqrt{-\Delta}}
\left ( 1 - \varphi \left (\frac{x}{t} \right ) \right ) |u(t)|^{p-1} u(t) dt
\right \|_{\dot H^{s_p + \sigma}} \lesssim 1, \label{prop41}
\end{align}
where the implied constant is uniform in $T_2 \in (0,T_+)$.  Indeed, the case $\sigma \in (0, 1- s_p]$ was covered 
in the proof of Proposition \ref{prop1}.  For $\sigma = 2 - s_p$, we first recall the estimate \eqref{prop42}
\begin{align*}
\| |u(t)|^{p-1} u(t) \|_{L^2(|x| \simeq |t|)} \lesssim |t|^{s_p-2},
\end{align*}
This estimate and the radial Sobolev embedding $\| |x|^{(d-2)/2} u \|_{L^\infty_x} \lesssim \| u \|_{\dot H^1}$
yield
\begin{align*}
\left \| \left ( 1 - \varphi \left ( \frac{x}{t} \right) \right )
 |u(t)|^{p-1}u(t) \right \|_{\dot H^1} 
\lesssim& \left \| \nabla \left ( 1 - \varphi \left ( \frac{x}{t} \right) \right )
 |u(t)|^{p-1} u(t) \right \|_{L^2}  \\
&+ \left \| \left ( 1 - \varphi \left ( \frac{x}{t} \right) \right )
 |u(t)|^{p-1} \nabla u(t) \right \|_{L^2} \\
\lesssim& |t|^{s_p- 3} 
+ |t|^{(1-p)(d-2)/2} \left \|
 ||x|^{(d-2)/2} u(t)|^{p-1} \nabla u(t) \right \|_{L^2}  \\
\lesssim& |t|^{s_p - 3} + |t|^{(1-p)(d-2)/2}. 
\end{align*}

Hence
\begin{align}
\int_{\delta}^{T_2} \left \| \frac{e^{it\sqrt{-\Delta}}}{\sqrt{-\Delta}} \left ( 1 - \varphi \left ( \frac{x}{t} \right) \right )
 |u(t)|^{p-1}u(t) \right \|_{\dot H^{2}} dt \lesssim 
\int_\delta^{+\infty} |t|^{s_p - 3} + |t|^{(1-p)(d-2)/2} dt \lesssim 1,
\end{align}
uniformly in $T_2$.  Interpolating this estimate with the known estimate for $\sigma \in (0,1-s_p]$, we obtain 
\eqref{prop41}.

Define $v$ as in the proof of Proposition \ref{prop2} and Proposition \ref{prop1}.  We now estimate $\| P_k v(0) \|_{\dot H^{s_p}}$.  By the Duhamel formula,
for any $T_- < T_1 < 0 < T_2 < T_+$, 
\begin{align*}
\la P_k v(0) &, P_k v(0) \ra_{\dot H^1} \\
=& \left \la P_k \left ( e^{iT_2 \sqrt{-\Delta}} v(T_2) + i \mu \int_0^{T_2} \frac{e^{it \sqrt{-\Delta}}}{\sqrt{-\Delta}} |u(t)|^{p-1}u(t) dt
\right ), \right. \\
& \left. P_k \left (  e^{iT_1 \sqrt{-\Delta}} v(T_1) - i \mu \int^0_{T_1} \frac{e^{i\tau \sqrt{-\Delta}}}{
\sqrt{-\Delta}} |u(\tau)|^{p-1}u(\tau) d\tau
\right )  \right \ra_{\dot H^{1}}.
\end{align*}
As in the proofs of Proposition \ref{prop2} and Proposition \ref{prop1}, we can conclude via Lemma \ref{lem1} that 
\begin{align*}
\la& P_k v(0), P_k v(0) \ra_{\dot H^{1}} = \\
&\lim_{T_1 \rightarrow T_-} \lim_{T_2 \rightarrow T_+}
\left \la 
P_k  \left ( \int_0^{T_2} \frac{e^{it \sqrt{-\Delta}}}{\sqrt{-\Delta}} |u(t)|^{p-1}u(t) dt \right ), 
P_k  \left ( \int^0_{T_1} \frac{e^{i\tau \sqrt{-\Delta}}}{\sqrt{-\Delta}} |u(\tau)|^{p-1}u(\tau) d\tau \right )
\right \ra_{\dot H^{1}}.
\end{align*}
We note that since we are localizing to frequencies at $2^k$, we are still able to apply Lemma \ref{lem1} to the $\dot H^1$ pairing rather than an $\dot H^{s_p}$ pairing. 
We split the Duhamel integrals into pieces and write the previous $\dot H^{s_p}$ pairing as a sum of pairings 
\begin{align*}
\la A + B, A' \ra + \la A , A' + B' \ra + \la B, B'\ra - \la A , A'\ra,
\end{align*}
where 
\begin{align}
A &:= P_k \left ( \int_0^{\delta} \frac{e^{it \sqrt{-\Delta}}}{\sqrt{-\Delta}} |u(t)|^{p-1}u(t) dt +
\int_{\delta}^{T_2} \frac{e^{it \sqrt{-\Delta}}}{\sqrt{-\Delta}} \left ( 1 - \varphi \left ( \frac{x}{t} \right) \right )
 |u(t)|^{p-1}u(t) dt \right ), \\
B &:= P_k \left (\int_{\delta}^{T_2} \frac{e^{it \sqrt{-\Delta}}}{\sqrt{-\Delta}} \varphi \left ( \frac{x}{t} \right ) |u(t)|^{p-1}u(t) dt \right ),
\end{align}
and $A',B'$ are the analogous integrals in the negative time direction.  

We now estimate $A$.  By \eqref{prop43} and \eqref{prop41} we have the estimate (with $\beta$ defined in Claim \ref{clm1})
\begin{align*}
\| A \|_{\dot H^{1}} \lesssim 
\eta^{p-1} \sum_j 2^{-\beta |j - k|} a_j + 2^{-\epsilon_p k} b_k
\end{align*}
where $b_k = 2^{-k(\sigma(k) - \epsilon_p)}$, and we are free to choose 
$\sigma(k) \in [0,1-s_p]$ for each $k$. We choose 
\begin{align*}
\sigma(k) = 
\begin{cases}
1 - s_p, &\mbox{if } k \geq 0, \\
0, &\mbox{if } k < 0,
\end{cases}
\end{align*}
so that 
\begin{align*}
b_k = 
\begin{cases}
2^{-k(1-s_p-\epsilon_p)} &\mbox{ if } k \geq 0, \\
2^{\epsilon k} &\mbox{ if } k < 0.
\end{cases}
\end{align*}
Since $\epsilon_p < 1 - s_p$, we see that $\{ b_k \} \in \ell^2$ and 
\begin{align}
\| \{ b_k \} \|_{\ell^2} \lesssim 1.
\end{align}
In summary, we have the estimate
\begin{align*}
\| A \|_{\dot H^{1}} \lesssim \eta^{p-1} \sum_j 2^{-\beta |j - k|} a_j + 2^{-\epsilon_p k} b_k, \quad 
\| \{b_k\}_k \|_{\ell^2} \lesssim 1.
\end{align*}
The same estimate holds with $A$ replaced by $A'$. This implies that 
\begin{align}
\la A , A' \ra \lesssim \eta^{2(p-1)} \left ( \sum_j 2^{-\beta |j - k|} a_j \right )^2+ 
2^{-2\epsilon_p k} b_k^2. \label{prop45}
\end{align}
 
We now estimate $\la A, A'+ B' \ra$.  By Lemma \ref{lem1}, 
\begin{align*}
- \mu i P_k \int_{T_1}^0 \frac{e^{i \tau \sqrt{-\Delta}}}{\sqrt{-\Delta}} |u(\tau)|^{p-1} u(\tau) d\tau \rightharpoonup P_k v(0)
\end{align*}
in $\dot H^{1}$ as $T_1 \rightarrow T_-$. As in the proof of Proposition \ref{prop2}, we have that 
\begin{align}
\left |\lim_{T_1 \rightarrow T_-} \lim_{T_2 \rightarrow T_+}
\la A , A' + B' \ra \right | 
&= \left |\lim_{T_1 \rightarrow T_-} \lim_{T_2 \rightarrow T_+} \left \la A , P_k v(0) \right \ra \right | \nonumber \\
&\lesssim a_k(0) \left ( 
\eta^{p-1} \sum_j 2^{-\beta |j - k|} a_j + 2^{-\epsilon_p k} b_k
\right ). \label{prop46}
\end{align}
By the same proof, the same estimate holds for $\la A + B , A' \ra$. 

Finally, we estimate the pairing $\la B, B' \ra$ which we write as 
\begin{align*}
\la &B , B' \ra \\ &= \int_{\delta}^{T_2} \int^{-\delta}_{T_1}
\left \la \varphi \left ( \frac{x}{|t|} \right ) |u(t)|^{p-1}u(t), P_k^2 e^{i(\tau - t)\sqrt{-\Delta}}
\varphi \left ( \frac{y}{|\tau|} \right ) |u(\tau)|^{p-1}u(\tau) \right \ra_{L^2} d\tau dt.  
\end{align*}
The operator $P_k^2 e^{i(\tau - t)\sqrt{-\Delta}}$ has a kernel given by 
\begin{align}
K_k ( x ) = K_k (|x|) = c \int_0^{\pi} \int_0^{+\infty} e^{i|x|\rho \cos \theta} e^{i(t- \tau)\rho} 
\psi^2 \left ( \frac{\rho}{2^k}\right ) \rho^{d-1} d \rho \sin^{d-2} \theta d\theta,
\end{align}
where  $\psi \in C^{\infty}_0$ is the Littlewood--Paley multiplier with $\supp \psi \subseteq
[1/2,2]$. As in the proof of Proposition \ref{prop1}, we have the estimate 
\begin{align}
|K_k(x - y)| \lesssim_L \frac{2^{dk}}{\la 2^k |\tau - t|\ra^L},
\end{align}
for every $L \geq 1$, which implies (as in the proof of Proposition \ref{prop1})
\begin{align*}
\la B, B' \ra \lesssim_L 
\begin{cases}
2^{(d - L)k}, &\mbox{if } k \geq 0, \\
1, &\mbox{if } k < 0,
\end{cases} 
\end{align*}
for all $L$ sufficiently large. Fixing $L$ sufficiently large, we obtain the estimate
\begin{align}
\la B, B' \ra \lesssim 
\begin{cases}
2^{-4k}, &\mbox{if } k \geq 0, \\
1, &\mbox{if } k < 0.
\end{cases} \label{prop44}
\end{align}

Combining \eqref{prop45}, \eqref{prop46}, and \eqref{prop44}, we obtain the estimate 

\begin{align*}
a^2_k(0) \lesssim& a_k(0) \left ( \eta^{p-1} \sum_j 2^{-\beta |j - k|} a_j + 2^{- \epsilon_p k} b_k \right )
+ \eta^{2(p-1)} \left ( \sum_j 2^{-\beta |j - k|} a_j \right )^2 \\ &+ 2^{-2 \epsilon_p k} b_k^2
+  \min \left ( 2^{-4k}, 1\right ).
\end{align*}
Hence, we have 
\begin{align*}
a_k(0) \lesssim& \eta^{p-1} \sum_j 2^{-\beta |j - k|} a_j + 2^{- \epsilon k} b_k +
\min \left ( 2^{-2k}, 1\right ).
\end{align*}
Using the definitions of $\alpha_k$ and $\alpha_k(0)$ and Claim \ref{clm1} we obtain
\begin{align*}
\alpha_k(0) &\lesssim \eta^{p-1} \alpha_k + \sum_j 2^{-\frac{\beta}{2}|j-k|} 2^{-\epsilon_p j} b_j 
+ \sum_j 2^{-\frac{\beta}{2}|j-k|} 2^{- j} c_j \\ &\lesssim
\eta^{p-1} \alpha_k(0) + \sum_j 2^{-\frac{\beta}{2}|j-k|} 2^{-\epsilon_p j} b_j 
+ \sum_j 2^{-\frac{\beta}{2}|j-k|} 2^{- \epsilon_p j} c_j
\end{align*}
where 
\begin{align*}
c_j := 
\begin{cases}
2^{-k(1-\epsilon_p)} &\mbox{ if } k \geq 0, \\
2^{\epsilon k} &\mbox{ if } k < 0.
\end{cases}
\end{align*}
By choosing $\eta$ sufficiently 
small, we have that 
\begin{align*}
\alpha_k(0) \lesssim \sum_j 2^{-\frac{\beta}{2}|j-k|} 2^{- \epsilon j} b_j +
 \sum_j 2^{-\frac{\beta}{2}|j-k|} 2^{-\epsilon j} c_j.
\end{align*}
Since $\epsilon_p \leq \beta/4$ we see by Minkowski's inequality that 
\begin{align*}
\| 2^{\epsilon_p k} \alpha_k(0) \|_{\ell^2} \lesssim \| \{b_k\} \|_{\ell^2} + 
\| \{c_k\} \|_{\ell^2} \lesssim 1.
\end{align*}
This finishes the proof of Proposition \ref{prop4}.
\end{proof}

\subsection{Rigidity argument for the soliton--like case}

We now prove the following rigidity result using a simple virial argument. 

\begin{ppn}\label{sol}
Let $u \in C(\R ; (\dot H^{s_p} \times \dot H^{s_p-1}) \cap (\dot H^1 \times L^2))$ be a global solution to \eqref{nlw} such that
the trajectory
\begin{align*}
\{ \vec u(t) : t \in \R \}
\end{align*}
is precompact in $(\dot H^{s_p} \times \dot H^{s_p-1}) \cap (\dot H^1 \times L^2)$.  Then $u = 0$.
\end{ppn}

\begin{proof}

Let $\varphi \in C^{\infty}_0(\R^d)$ be radial with $\varphi = 1$ for $|x| \leq 1$ and $\varphi = 0$ if $|x| \geq 2$.  For $R > 0$,
denote $\varphi_R = \varphi(\cdot / R)$.  We define
\begin{align}
\rho(R) = \sup_{t \in \R} \int_{|x| \geq R} |\nabla u(t)|^2 + |\partial_t u(t)|^2 + \frac{|u(t)|^2}{|x|^2} + |u(t)|^{p+1} dx.
\end{align}
Since $\{ \vec u(t) : t \in \R \}$ is precompact in $(\dot H^{s_p} \times \dot H^{s_p-1}) \cap (\dot H^1 \times L^2)$, we have by
Sobolev embedding and Hardy's inequality that $\rho(R) < +\infty$ and
$$
\lim_{R \rightarrow +\infty} \rho(R) = 0.
$$
Using \eqref{nlw} and integration by parts, it is standard to verify the two virial identities
\begin{align}
\frac{d}{dt} \int \varphi_R u(t) \partial_t u(t) dx =&
\int |\partial_t u(t)|^2 dx - \int |\nabla u(t)|^2 dx - \mu \int |u(t)|^{p+1} dx \label{vir1} \\ &+ O(\rho(R)), \nonumber\\
\frac{d}{dt} \int \varphi_R x \cdot \nabla u(t) \partial_t u(t) dx =&
- \frac{d}{2} \int |\partial_t u(t)|^2 dx + \frac{d-2}{2} \int |\nabla u(t)|^2 dx \label{vir2}\\ &+ \mu \frac{d}{p+1} \int |u(t)|^{p+1} dx
 + O(\rho(R)), \nonumber
\end{align}
where the big--oh terms are uniform in $t$.  Indeed, for \eqref{vir1} we have 
\begin{align*}
\frac{d}{dt} \int \varphi_R u(t) \partial_t u(t) dx =& 
\int \varphi_R | \partial_t u(t)|^2 dx + \int \varphi_R u(t) \partial^2_t u(t) dx \\
=& \int \varphi_R | \partial_t u(t)|^2 dx + \int \varphi_R u(t) \Delta u(t) dx + - \mu\int \varphi_R | u(t) |^{p+1} dx \\
=& \int \varphi_R | \partial_t u(t)|^2 dx - \int \varphi_R |\nabla u(t)|^2 dx - \int u(t) \nabla \varphi_R \cdot \nabla u(t) dx \\
&- \mu \int \varphi_R | u(t) |^{p+1} dx \\
=& \int |\partial_t u(t)|^2 dx - \int |\nabla u(t)|^2 dx - \mu \int |u(t)|^{p+1} dx + A_R(t)
\end{align*}
where 
\begin{align*}
A_R(t) := \int (\varphi_R - 1) \left ( |\partial_t u(t)|^2 - |\nabla u(t) |^2 - \mu |u(t)|^{p+1} \right ) dx 
- \int u(t) \nabla \varphi_R \cdot \nabla u(t) dx.
\end{align*}
The first integral in the definition of $A_R(t)$ is easily seen to be 
$O(\rho(R))$.  Since $|\nabla \varphi_R(x)| \lesssim |x|^{-1} \chi_{\{|x| \simeq R\}}$, we have by Cauchy-Schwarz and Hardy's inequality
\begin{align*}
\int |u(t)| |\nabla \varphi_R | |\nabla u(t)| dx \leq C \rho(R).
\end{align*}
Thus, $|A_R(t)| \leq C \rho(R)$ uniformly in $t$.  The calculation for 
\eqref{vir2} is similar and is omitted. 

We treat the cases $\mu = -1$ and $\mu = 1$ separately.

\subsubsection*{Case 1: $\mu = -1$}
Let
\begin{align}
y_R(t) = \frac{d-2}{2} \int \varphi_R u(t) \partial_t u(t) dx + \int \varphi_R x \cdot \nabla u(t) \partial_t u(t) dx.
\end{align}
Note that by Hardy's inequality and the compactness of the trajectory, we have the uniform in $t$ bound
\begin{align}
|y_R(t)| &\leq \frac{d-2}{2} \int |\varphi_R ||u(t)||\partial_t u(t)| dx + \int |\varphi_R x| |\nabla u(t)| |\partial_t u(t)| dx \nonumber \\
&\leq C R \| \vec u(t) \|_{\dot H \times L^2}^2 \nonumber \\
&\leq C R. \label{sol1}
\end{align}
By the virial identities \eqref{vir1} and \eqref{vir2}, we have that
$$
y_R'(t) = -\int |\partial_t u(t)|^2 dx - d \left (\frac{1}{p+1} - \frac{d-2}{2d} \right ) \int |u(t)|^{p+1} dx
+ O(\rho(R)).
$$
where the big--oh term is uniform in $t$.  Since $s_p < 1$, we have that $p < \frac{d+2}{d-2}$ so $\frac{1}{p+1} - \frac{d-2}{2d} > 0$
and 
$$
\int |u(t)|^{p+1} dx \leq C \left ( -y_R'(t) + O(\rho(R)) \right ).
$$
Integrating the previous expression from $0$ to $T$ and dividing by $T$ we have
\begin{align*}
\frac{1}{T} \int_0^T \int |u(t)|^{p+1} dx dt &\leq C \frac{1}{T} |y_R(T) - y_R(0)| + C \rho(R) \\
&\leq  C \frac{R}{T} + C \rho(R).
\end{align*}
Setting $R = \sqrt{T}$, we obtain
\begin{align}
\lim_{T \rightarrow +\infty} \frac{1}{T} \int_0^T \int | u(t)|^{p+1} dx dt = 0. \label{sol2}
\end{align}

We claim that there exists a sequence of natural numbers $\{k_n\}_n$ such that
\begin{align}
\lim_{n \rightarrow \infty} \int_{k_n}^{k_n + 1}  \int |u(t)|^{p+1} dx dt = 0. \label{sol3}
\end{align}
Suppose not. Then there exists $\epsilon > 0$ so that for all $k \geq 0$
\begin{align*}
\int_{k}^{k+1} \int |u(t)|^{p+1} dx dt \geq \epsilon.
\end{align*}
Summing the previous expression over all $k$ from $0$ to $N-1$ implies
\begin{align*}
\int_{0}^{N} \int |u(t)|^{p+1}dx dt \geq N \epsilon, \quad \forall N \geq 1,
\end{align*}
which contradicts \eqref{sol2}.  This proves \eqref{sol3}.

Denote $t_n = k_n$ and write \eqref{sol3} as
\begin{align}
\lim_{n \rightarrow \infty} \int_0^1 \int |u(t_n + t)|^{p+1} dx dt = 0. \label{sol4}
\end{align}
By compactness of the trajectory, there exists $(U_0,U_1) \in (\dot H^{s_p} \times \dot H^{s_p-1}) \cap (\dot H^1 \times L^2)$
such that
$$
(u(t_n), \partial_t u(t_n)) \rightarrow (U_0,U_1) \quad \mbox{in } (\dot H^{s_p} \times \dot H^{s_p-1}) \cap (\dot H^1 \times L^2).
$$
Let $U(t)$ be the solution to \eqref{nlw} with initial data $(U_0,U_1)$.  By the local theory for \eqref{nlw}, for $0 < t_0 < 1$ sufficiently small,
$$
\lim_{n \rightarrow \infty} \sup_{t \in [0,t_0]} \| \vec u(t_n + t) - \vec U(t) \|_{\energysp} = 0.
$$
This fact, Fatou's lemma, and \eqref{sol4} imply that
$$
\int_0^{t_0} \int | U(t)|^{p+1} dx dt = 0,
$$
Thus, $U = 0$ so that
$$
\lim_{t \rightarrow +\infty} \| \vec u(t) \|_{(\energysp) \cap (\dot H^1 \times L^2)} = 0.
$$
By the small data theory $u \equiv 0$.

\subsubsection*{Case 2: $\mu = 1$}  Similar to the focusing case, we define
\begin{align}
y_R(t) = \frac{d-1}{2} \int \varphi_R u(t) \partial_t u(t) dx + \int \varphi_R x \cdot \nabla u(t) \partial_t u(t) dx,
\end{align}
and note that
\begin{align*}
|y_R(t)| \leq C R
\end{align*}
for all $t \in \R$.  Using the virial identities we have that
\begin{align*}
y_R'(t) =& -\frac{1}{2} \int |\partial_t u(t)|^2 dx - \frac{1}{2} \int |\nabla u(t)|^2 dx \\ &-\left ( \frac{d-1}{2}
- \frac{d}{p+1} \right )\int |u(t)|^{p+1} dx
+ O(\rho(R)) \\
=& - E(\vec u(t)) - \left (  \frac{d-1}{2} - \frac{d+1}{p+1} \right ) \int |u(t)|^{p+1} dx + O(\rho(R)).
\end{align*}
The assumption $s_p \geq 1/2$ implies that $\frac{d-1}{2} - \frac{d+1}{p+1} \geq 0$.  Hence
\begin{align}
E(\vec u(t)) \leq - y'_R(t) + O(\rho(R)).
\end{align}
Since $\mu = 1$, we have by conservation of energy that
$$
\| \vec u(0) \|_{\dot H^1 \times L^2}^2 \leq 2 E(\vec u(0)) = 2 E(\vec u(t)) \leq - 2y'_R(t) + O(\rho(R)).
$$
Integrating the previous expression from $0$ to $T$, dividing by $T$, and choosing $R = \sqrt{T}$, we obtain
\begin{align*}
\| \vec u(0) \|_{\dot H^1 \times L^2}^2 \leq C \left ( \frac{1}{\sqrt{T}} + \rho(\sqrt{T}) \right ) \rightarrow 0
\end{align*}
as $T \rightarrow +\infty$.  Hence $\| \vec u(0) \|_{\dot H^1 \times L^2} = 0$ so that $u = 0$.
\end{proof}

\begin{proof}[Proof of Theorem \ref{thm1}]
If the conclusion of Theorem \ref{thm1} was false, then by Proposition \ref{contra} there exists a nonzero solution $u$ to 
\eqref{nlw} with the compactness property.  By Proposition \ref{prop2} and Proposition \ref{noenergycas}
we see that the scaling parameter must satisfy $N(t) \equiv 1$. But by Proposition \ref{prop4} and 
Proposition \ref{sol}, we have that $u = 0$, which is a contradiction.  Thus, the conclusion of Theorem \ref{thm1} must hold. 
\end{proof}

\end{document}